\documentclass[12pt,a4paper]{article}
\usepackage[utf8]{inputenc}
\usepackage{amsmath,amssymb,amsthm,amsfonts,amstext,amsbsy,amsopn,amscd}
\usepackage{amsxtra,latexsym,exscale,verbatim}
\usepackage{bm,bbm,paralist}
\usepackage{graphicx}
\usepackage{yfonts,color,hyperref}
\usepackage[normalem]{ulem}

\newtheorem{theorem}{Theorem}[section]
\newtheorem{lemma}[theorem]{Lemma}

\newtheorem{rem}[theorem]{Remark}

\newtheorem{assump}[theorem]{Assumptions}

\numberwithin{equation}{section}

\makeatletter
\newcommand{\xrightharpoonup}[2][]{\ext@arrow 0359{%
\arrowfill@\relbar\relbar\rightharpoonup}{#1}{#2}}
\makeatother

\DeclareMathOperator{\DIV}{div}

\DeclareMathOperator{\T}{\mathcal{T}_\eps}
\DeclareMathOperator{\Tb}{\mathcal{T}_\eps^\mathrm{b}}
\DeclareMathOperator{\F}{\mathcal{F}_\eps}

\DeclareMathOperator{\Q}{\mathcal{Q}_\eps}
\DeclareMathOperator{\G}{\mathcal{G}_\eps}

\DeclareMathOperator{\wto}{\rightharpoonup}

\newcommand{\oo}{\Omega}
\newcommand{\rr}{\mathbb{R}}

\newcommand{\dd}{\,\mathrm{d}}
\newcommand{\eps}{\varepsilon}

\newcommand{\nablaY}{\nabla_{\!y}}
\newcommand{\nablaD}{\nabla^\delta}

\newcommand{\rmH}{\mathrm{H}}
\newcommand{\Hh}{\rmH^1(\oo)}

\newcommand{\rmL}{\mathrm{L}}
\newcommand{\wh}{\widehat}
\newcommand{\wt}{\widetilde}

\def\Xint#1{\mathchoice
{\XXint\displaystyle\textstyle{#1}}%
{\XXint\textstyle\scriptstyle{#1}}%
{\XXint\scriptstyle\scriptscriptstyle{#1}}%
{\XXint\scriptscriptstyle\scriptscriptstyle{#1}}%
\!\int}
\def\XXint#1#2#3{{\setbox0=\hbox{$#1{#2#3}{\int}$}
\vcenter{\hbox{$#2#3$}}\kern-.5\wd0}}

\def\dashint{\Xint-}

\begin{document}
\title{Corrector estimates for a thermo-diffusion model\\with weak thermal coupling}
\author{Adrian Muntean\footnote{Department of Mathematics and Computer Science, Karlstads Universitet, Sweden} \quad and \quad
Sina Reichelt\footnote{Weierstra\ss-Institut f\"ur Angewandte Analysis und Stochastik (WIAS), Berlin, Germany (corresponding author)}}
\maketitle

\begin{abstract}\noindent The present  work deals with the derivation of corrector estimates for the two-scale homogenization of a thermo-diffusion model with weak thermal coupling posed in a heterogeneous medium endowed with periodically arranged high-contrast microstructures. The terminology ``weak thermal coupling'' refers here to the variable scaling in terms of the small homogenization  parameter $\eps$ of the heat conduction-diffusion interaction terms, while  the ``high-contrast'' is thought  particularly in terms  of the heat conduction properties of the composite material. As main target, we justify the first-order terms of the multiscale asymptotic expansions in the presence of coupled fluxes, induced by the joint contribution of Sorret and Dufour-like effects. The contrasting heat conduction combined with cross coupling lead to the main mathematical difficulty in the system. Our approach relies on the method of periodic unfolding  combined with $\eps$-independent estimates for the thermal and concentration fields and for their coupled fluxes.   

\end{abstract}



\section{Introduction}

This paper deals with the justification of the two-scale asymptotic expansions method applied to a thermo-diffusion problem arising in the context of transport of densities of hot colloids in media made of periodically-distributed microstructures. 
Following \cite{KAM2014}, we study a system of two coupled semi-linear parabolic equations, where the diffusivity for the concentration $u_\eps$ is of order $O(1)$ and for the temperature $\theta_\eps$ it is of order $O(\eps^2)$. Here $\eps > 0$ denotes the characteristic length scale of the underlying microstructure. We rigorously justify the expansions $u_\eps (x) \approx u(x) + \eps U(x,x/\eps)$ and $\theta_\eps(x) \approx \Theta(x,x/\eps)$ and prove an estimate of the type
\begin{align}
\label{main}
&\Vert \T u_\eps {-} u \Vert_{\rmL^\infty(0,T; \rmL^2(\oo\times Y_*))}
+ \Vert \T(\nabla u_\eps) {-} (\nabla u {+} \nablaY U) \Vert_{\rmL^2(0,T; \rmL^2(\oo\times Y_*))} \nonumber \\
& + \Vert \T \theta_\eps {-} \Theta \Vert_{\rmL^\infty(0,T; \rmL^2(\oo\times Y_*))}
+ \Vert \T(\eps\nabla \theta_\eps) {-} \nablaY \Theta \Vert_{\rmL^2(0,T; \rmL^2(\oo\times Y_*))}
\leq \sqrt{\eps}C ,
\end{align}
where $\oo \subset \rr^d$ denotes the macroscopic domain and $Y_* \subset [0,1)^d$ is the perforated reference cell. 
Estimate \eqref{main} basically gives a quantitative indication of the speed of the (two-scale) convergence between the unknowns of our problem and their limits, which is detailed in the forthcoming sections. This work follows up previous successful attempts of deriving quantitative corrector estimates using periodic unfolding; see e.g. \cite{Gri04,Gris05,OV07,FMP12,Diss-SR,Rei2016}. The unfolding technique allows for homogenization results under minimal regularity assumptions on the data and on the choice of allowed microstructures. The novelty we bring in here is the combination of three aspects: (i) the asymptotic procedure refers to a suitably perforated domain, (ii) presence of a cross coupling in gradient terms, and (iii) lack of compactness for $\theta_\eps$. Our working techniques combines $\eps$-independent a priori estimates for the solutions and periodic unfolding-based estimates such as the \emph{periodicity defect} in \cite{Gri04} and the \emph{folding mismatch} in \cite{Rei2016}. Estimate \eqref{main} improves existing convergence rates for semi-linear parabolic equations with possibly non-linear boundary conditions in \cite{FMP12} or small diffusivity in \cite{Diss-SR} from $\eps^{1/4}$ to $\eps^{1/2}$. 
This improvement is obtained by studying all equations in the two-scale space $\oo\times Y_*$ and by suitably rearranging and controlling occurring error terms $\Delta^\eps_{\mathrm{error}}$.

It is worth noting that the availability of corrector estimates for the thermo-diffusion system allows in principle the construction of rigorously convergent multiscale numerical methods (for instance based on MsFEM like in \cite{BLL14}) to capture thermo-diffusion effects in porous media. Interestingly, for the thermo-diffusion system posed in perforated domains such convergent multiscale numerical methods are yet unavailable.

\medskip
The paper is structured as follows: In Section \ref{sec:model}, we introduce the thermo-diffusion model and prove existence as well as a priori estimates for the solutions of the microscopic problem respective the two-scale limit problem. The periodic unfolding method and auxiliary corrector estimates are presented in Section \ref{subsec:two-scale} and \ref{subsec:auxiliary}, respectively. Finally, the corrector estimates in \eqref{main}  are proved in Section \ref{subsec:proof}. We conclude our paper with a discussion  in Section \ref{sec:discuss}.

\section{A thermo-diffusion model}
\label{sec:model}

\subsection{Model equations. Notation and assumptions}

We investigate a system of reaction-diffusion equations which includes mollified cross-diffusion terms and different diffusion length scales. The cross-diffusion terms are motivated by the incorporation of Soret and Dufour effects as outlined in \cite{KAM2014}. For more information on phenomenological descriptions of thermo-diffusion, we refer the reader to \cite{GrM84}. The concentrations of the transported species through the perforated domain $\oo_\eps$ are denoted by $u_\eps$, while  $\theta_\eps$ is the temperature. The overall interplay between transport and reaction is modeled here by the following system of partial differential equations:
\begin{equation}
\label{eq:prob-eps}
\begin{array}{rcll}
\dot{u}_\eps & = & \DIV (d_\eps \nabla u_\eps) + \tau \eps^\alpha \nabla u_\eps \cdot \nablaD \theta_\eps + R(u_\eps) & \text{ in } \oo_\eps \\
\dot{\theta}_\eps & = & \DIV(\eps^2 \kappa_\eps \nabla \theta_\eps) + \mu \eps^\beta \nabla \theta_\eps \cdot \nablaD u_\eps  & \text{ in } \oo_\eps 
\end{array}
\end{equation}
supplemented with the Neumann boundary conditions 
\begin{align}
\label{eq:prob-eps-bdry}
\begin{array}{lcll}
- d_\eps \nabla u_\eps \cdot \nu & = & \eps(a u_\eps + b v_\eps) & \text{ on } \partial T_\eps \\
- \eps^2 \kappa_\eps \nabla \theta_\eps \cdot \nu & = & \eps g \theta_\eps & \text{ on } \partial T_\eps \\
- d_\eps \nabla u_\eps \cdot \nu & = & 0 & \text{ on } \partial\oo_\eps \backslash \partial T_\eps \\
- \eps^2 \kappa_\eps \nabla \theta_\eps \cdot \nu & = & 0 & \text{ on } \partial\oo_\eps \backslash \partial T_\eps
\end{array}
\end{align}
and the initial conditions
\begin{align}
\label{eq:prob-eps-init}
u_\eps (0,x) = u^0_\eps(x) \quad\text{and}\quad \theta_\eps (0,x) = \theta^0_\eps(x), x\in   \oo_\eps.
\end{align}
First of all, it is important to note that the $\eps$-scaling for some of the terms in the system is variable with $\alpha, \beta \geq 0$. We refer to  the suitably scaled heat conduction-diffusion interaction terms $\eps^\alpha \nabla u_\eps \cdot \nablaD \theta_\eps$ and $\eps^\beta \nabla u_\eps \cdot \nablaD \theta_\eps$ as ``weak thermal couplings'', while  the ``high-contrast'' is thought here particularly in terms  of the heat conduction properties of the composite material that can be seen in $\eps^2 \kappa_\eps \nabla \theta_\eps$.
In this context, $\nu$ denotes the normal outer unit vector of $\oo_\eps$. The matrix $d_\eps $ is the diffusivity associated to the concentration of the (diffusive) species $u_\eps$, $\kappa_\eps$ is the heat conductivity, while $\tau_\eps:= \tau \eps^\alpha$ and  $\mu_\eps:=\mu \eps^\beta$ are the Soret and Dufour coefficients. Note that $d_\eps $, $\kappa_\eps$, $\tau$, and $\mu$ are either positive definite matrices, or they are positive real numbers.  Furthermore, the reaction term $R(\cdot)$ models the Smoluchovski interaction production.  
In  the original model from \cite{KAM2014}, the function $v_\eps$ is an additional unknown modeling the mass of deposited species on the pore surface $\Gamma_\eps$, and it is shown to possess the regularity $v_\eps \in \rmH^1(0,T; \rmL^2(\Gamma_\eps)) \cap \rmL^\infty((0,T) \times \Gamma_\eps)$. Here we assume $v_\eps$ as given data. 
We point out that the linear boundary terms are relevant for the regularity of solutions, but that they are not required to prove the convergence rate of order of $\sqrt{\eps}$ in \eqref{main}.

To deal with perforated domains we employ the method of periodic unfolding as presented in \cite{CDZ06}. Let $Y = [0,1)^d$ denote the standard unit-cell. We fix here and for all the following assumptions on the domain and the microstructure. 
\begin{assump}
\label{assump:domain} Our geometry is designed as follows:
\begin{compactenum}
\item[(i)] The domain $\oo = \prod_{i=1}^d \, [ 0, l_i )$ is a $d$-polytope with length $l_i > 0$ for all $1 \leq i \leq d$.
\item[(ii)] The reference hole $T \subseteq Y$ is an open Lipschitz domain and the perforated cell $Y_* := Y\backslash \overline{T}$ satisfies $Y_* \neq \emptyset$. Moreover $Y_*$ is a connected Lipschitz domain and $\partial Y_* \cap \partial Y$ is identical on all faces of $Y$ .
\end{compactenum}
\end{assump}
The set of all nodal points is given via $N_\eps := \lbrace \xi \in \mathbb{Z}^d \,|\, \eps(\xi + Y) \subseteq \oo \rbrace$. 
With this we define the pore part $T_\eps $ and the perforated domain $\oo_\eps$, which is connected, via
\begin{align}
\label{eq:def-domain}
T_\eps : = \bigcup_{\xi \in N_\eps} \eps(\xi + T)
\quad\text{and}\quad
\oo_\eps :=\bigcup_{\xi \in N_\eps} \eps(\xi + Y_*^\circ) ,
\end{align}
where $A^\circ$ denotes the interior of the set $A$. 
Both sets are open and form together the original domain  $\overline{\oo} = \overline{T}_\eps \cup \overline{\oo}_\eps$.   
\begin{figure}[t]
\label{fig:cell}
\centering
\includegraphics[width=5cm]{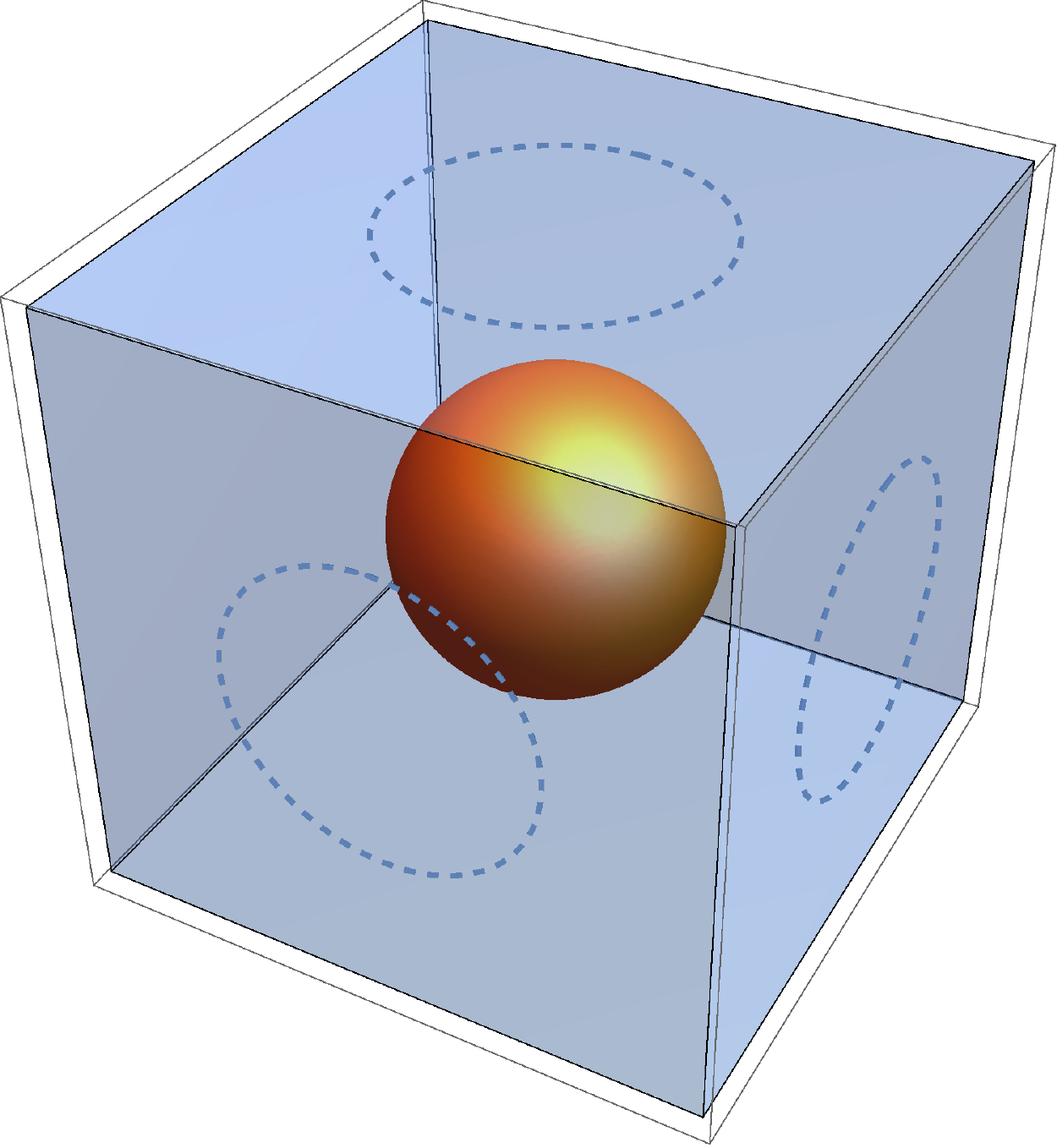}
\includegraphics[width=5cm]{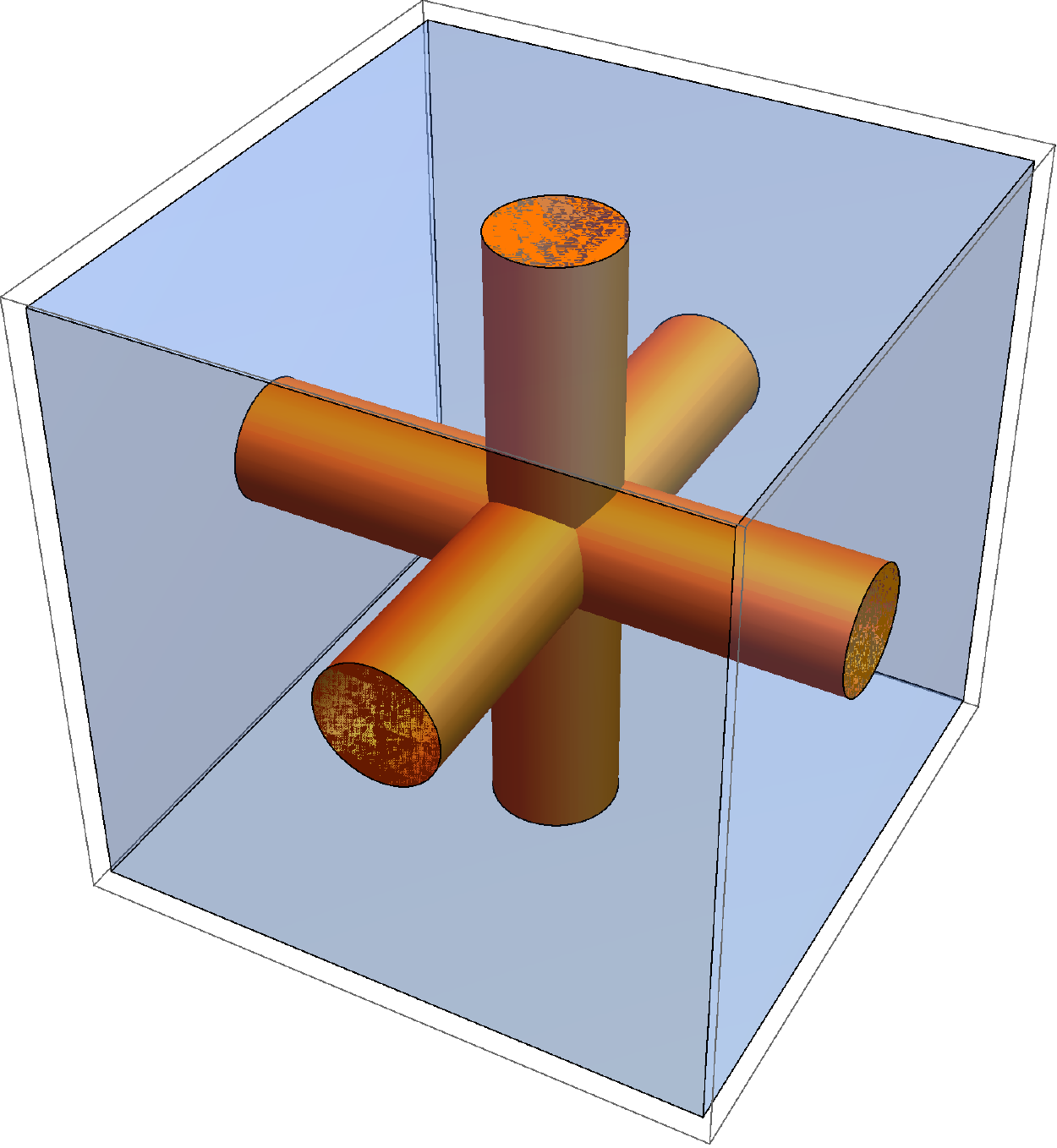}
\includegraphics[width=5cm]{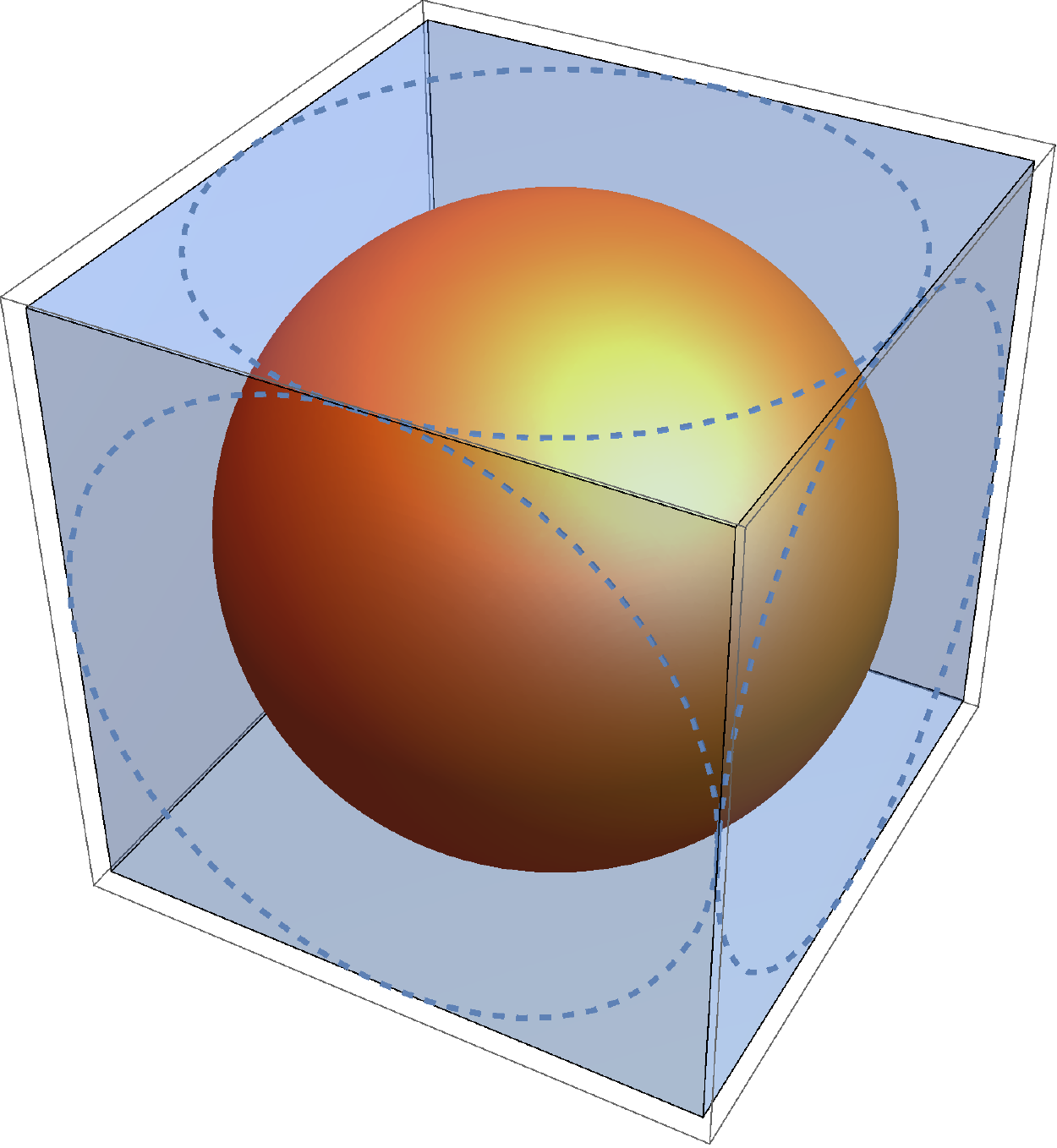}
\caption{Example of microstructures: (i) isolated inclusion; (ii) pipe structure; (iii) holes touching the boundary are not admissible.}
\end{figure}
The assumptions on the domain guarantee the existence of suitable extensions from $\oo_\eps$ to $\oo$ (cf.\ Theorem \ref{thm:extension}). Also traces exist and are well-defined on the boundaries $\partial \oo_\eps$ and $\partial T$. 
With this, perforated domains with isolated holes as well as the prominent ``pipe-model'' for porous media are included in our considerations, see Figure \ref{fig:cell}. The boundary of the perforated domain $\oo_\eps$ is given by $\partial\oo_\eps  = \left( \partial\oo \cup \partial T_\eps \right) \backslash \left( \partial\oo \cap \partial T_\eps \right)$. Indeed, intersected pore structures at the boundary $\partial\oo \cap \partial T_\eps \neq \emptyset$ as in Figure \ref{fig:cell}(ii) are not excluded.

\begin{rem}
\label{rem:boundary}
In the following we denote by $\eps$ a sequence $( \eps_n)_{n\in\mathbb{N}}$ of numbers satisfying $\eps_n^{-1} \in \mathbb{N}$. This implies that all microscopic cells $\eps(\xi + Y_*)$, for $\xi\in \mathbb{Z}^d$, are contained in $\oo_\eps$ and no intersected cells occur at the boundary $\partial\oo$. 

This assumption (tremendously) simplifies the presentation in this paper, however, we believe that the same results can be obtained for Lipschitz domains $\oo$ by considering a bigger $d$-polytope $\widetilde{\oo}_\eps$ with $\oo_\eps \subseteq \widetilde{\oo}_\eps$. Then, all relevant coefficients, functions, and solutions are suitably extended from $\oo_\eps$ to $\widetilde{\oo}_\eps$.
\end{rem}

\begin{assump}
\label{assump:data} We impose the following restrictions on the data:

\begin{compactenum}
\item[(i)]  The diffusion matrices $d_\eps$ and $\kappa_\eps$ are given via 
\begin{align*}
d_\eps(x) := \mathbb{D}( \tfrac{x}{\eps})
\quad\text{and}\quad
\kappa_\eps(x) := \mathbb{K}( \tfrac{x}{\eps}),
\end{align*}
where $\mathbb{D}, \mathbb{K} \in \rmL^\infty(Y_*; \rr^{d\times d}_\mathrm{sym})$ are symmetric and uniformly elliptic, i.e.\ 
\begin{align*}
\exists\, C_\mathrm{elip} > 0 ,\, \forall\, (\xi,y) \in \rr^d \times Y_* :\quad
C_\mathrm{elip} |\xi|^2 \leq \mathbb{D}(y) \xi \cdot \xi \leq C_\mathrm{elip}^{-1} |\xi|^2 .
\end{align*}
\item[(ii)]  The constants $\tau,\mu, a,b,g$ are non-negative.
\item[(iii)]  The reaction term $R: \rr \to \rr$ is globally Lipschitz continuous, i.e.\ 
\begin{align*}
\exists\, L>0, \forall\, s_1,s_2 \in \rr: \quad
|R(s_1) {-} R(s_2)| \leq L|s_1 {-} s_2| .
\end{align*}
Moreover, it is $R(s) = 0$ for all $s < 0$. 
\item[(iv)]  The sink/source term $v_\eps$ is given via $v_\eps(t,x) := \mathbb{V}(t,x,x/\eps)$ for any data $\mathbb{V} \in \mathrm{C}([0,T]; (\mathrm{W}^{1,\infty}(\oo; \rmL^2(\partial T)))$. 
\end{compactenum}
\end{assump}
Here we denote with $a \cdot b$ the scalar product of vectors in $\rr^d$ and set 
\begin{align*}
\rmL^\infty_+ (\oo) := \lbrace \varphi \in \rmL^\infty(\oo) \,|\, \varphi \geq 0 \text{ a.e.\ in } \oo\rbrace .
\end{align*}
For technical reasons, we introduce the mollified gradient $\nablaD$ which is given as follows: 
for $\delta>0$, we introduce the mollifier
\begin{align*}
J_\delta (x) := \left\lbrace
\begin{array}{ll}
C \exp(1/(|x|^2 - \delta^2)) \quad & \text{if } |x| < \delta ,\\
0 & \text{if } |x| \geq \delta ,
\end{array}\right.
\end{align*}
where the constant $C>0$ is selected such that $\int_{\rr^d} J_\delta \dd x = 1$. Using $J_\delta$ we define for $u \in \rmL^1(\rr^d)$ the mollified gradient
\begin{align*}
\nablaD u := \nabla \left[  \int_{B(x,\delta)} J_\delta(x {-}\xi) u(\xi) \dd \xi \right],
\end{align*}
where $B(x,\delta)$ denotes the ball centered at $x\in \rr^d$ with radius $\delta$. According to \cite[Sec.~C.4]{Eva98} there holds $\nablaD u\in \mathrm{C}^\infty(\oo)$ and 
\begin{align}
\label{eq:est-delta}
\exists\,C_\delta>0 , \forall\, u \in \rmL^2(\oo):\quad 
\Vert \nablaD u \Vert_{\rmL^\infty(\oo)} \leq C_\delta \Vert u \Vert_{\rmL^2(\oo)} .
\end{align}
We assume throughout this text that $\eps$ and $\delta$ are chosen such that $\delta > 2\eps\mathrm{diam}(Y)$ holds. This assumption arises in Lemma \ref{lemma:prop-T-2}. 

\subsection{Existence of solutions and {\em a priori} estimates}

Now, let us consider the case $\alpha {=} \beta {=} 1$. This subsection and the next one are devoted the existence of weak solutions to our target problem. 

\begin{theorem}
\label{thm:exist-eps}
Let the Assumptions \ref{assump:domain} and \ref{assump:data} hold and let the initial condition $(u_\eps^0, \theta_\eps^0)$ satisfy 
\begin{align*}
\exists\, C_0, M_0  >0 : \quad
& \Vert u_\eps^0 \Vert_{\rmH^1(\oo_\eps)} +
\Vert \theta_\eps^0 \Vert_{\rmL^2(\oo_\eps)} +
\eps \Vert \nabla \theta_\eps^0 \Vert_{\rmL^2(\oo_\eps)}
\leq C_0 , \\
& 0 \leq u_\eps^0(x), \theta_\eps^0(x) \leq M_0 \quad\text{for a.e.\ } x\in \oo_\eps .
\end{align*}
Then there exists for every $\eps>0$ a unique solution $(u_\eps,\theta_\eps)$ of \eqref{eq:prob-eps}--\eqref{eq:prob-eps-init} with
\begin{align*}
u_\eps, \theta_\eps \in \rmH^1(0,T; \rmL^2(\oo_\eps)) \cap \rmL^\infty(0,T; \rmH^1(\oo_\eps)) \cap \rmL^\infty_+((0,T)\times \oo_\eps) .
\end{align*}
Moreover the solution is non-negative, i.e.\ $0 \leq u_\eps, \theta_\eps \leq M$ almost everywhere in $[0,T]\times \oo_\eps$, and uniformly bounded
\begin{align}
\label{eq:sol-bound}
\begin{array}{l}
\Vert u_\eps \Vert_{\rmH^1(0,T; \rmL^2(\oo_\eps))} + 
\Vert \nabla u_\eps \Vert_{\rmL^\infty(0,T; \rmL^2(\oo_\eps))} \\
+ \Vert \theta_\eps \Vert_{\rmH^1(0,T; \rmL^2(\oo_\eps))} + 
\eps \Vert \nabla \theta_\eps \Vert_{\rmL^\infty(0,T; \rmL^2(\oo_\eps))} \leq C ,
\end{array}
\end{align}
where the constants $M,C>0$ are independent of $\eps$. 
\end{theorem}
\begin{proof}
The existence of solutions, non-negativity, and uniform boundedness follow from the Lemmata 3.2 --3.6 and Theorem 3.8 in \cite{KAM2014} by replacing $\kappa_\eps$ and $\tau_\eps$ with $\eps^2\kappa_\eps$ and $\eps\tau$, respectively. 
%\colR
Note that the proof can be generalized from diffusion coefficients $d_\eps, \kappa_\eps \in \rr$ to symmetric matrices as in Assumption \ref{assump:data}(i). In equation (35) respective (57) in \cite{KAM2014} it holds for $A \in \rmL^\infty(\oo;\rr^{d \times d}_\mathrm{sym})$ and $u \in \Hh$: 
\begin{align}
\label{eq:sym-trick}
\frac{\dd}{\dd t} \int_\oo A \nabla u \cdot \nabla u \dd x = 
\int_\oo A \nabla \dot{u} \cdot \nabla u \dd x +
\int_\oo A \nabla u \cdot \nabla \dot{u} \dd x \stackrel{A = A^T}{=} 
2  \int_\oo A \nabla \cdot \nabla \dot{u} \dd x .
\end{align}
This argumentation also requires linear boundary terms. Otherwise one has to argue as in \cite[Thm.~3.2]{Tem88} or \cite[Prop.~1]{MRT14} and differentiate the whole equation with respect to time and then use a second Gr\"onwall argument.
\end{proof}

\begin{rem}
Since our solutions are uniformly bounded in $\rmL^\infty((0,T) \times \oo_\eps)$, we may consider reaction terms with arbitrary growth as in \cite{KAM2014}. Also note that estimate \eqref{eq:sol-bound} remains valid for all $\alpha, \beta \geq 1$ and $\beta = 0$. 
\end{rem}

\subsection{The two-scale limit system}

For the parameters $\alpha {=} \beta {=} 1$,  we obtain in the limit $\eps\to0$ the following two-scale system
\begin{align}
\label{eq:prob-limit}
\begin{array}{rcll}
\dot{u} & = & \DIV( d_\mathrm{eff} \nabla u) + R (u) + \frac{|\partial T|}{|Y_*|} (a u + b v_0) & \text{ in } \oo \\
\dot{\Theta} & = & \DIV_{\!y} ( \mathbb{K} \nablaY \Theta) + \mu \nablaY \Theta \cdot \nablaD u  & \text{ in } \oo \times Y_*
\end{array}
\end{align}
supplemented with the boundary conditions 
\begin{align}
\label{eq:prob-limit-bdry}
\begin{array}{lcll}
- d_\mathrm{eff} \nabla u \cdot \nu & = & 0 & \text{ on } \partial\oo \\
- \mathbb{K} \nablaY \Theta \cdot \nu_{Y_*} & = & g \Theta & \text{ on }  \oo\times \partial T \\
\Theta \text{ is periodic } & & & \text{ on }  \oo \times \partial Y 
\end{array}
\end{align}
and the initial conditions
\begin{align}
\label{eq:prob-limit-init}
u(0,x) = u^0(x) \quad\text{and}\quad \Theta(0,x,y) = \Theta^0(x,y) .
\end{align}
Here $\nu$ and $\nu_{Y_*}$ denote the normal outer unit vector of $\oo$ and $Y_*$, respectively. 
To capture the oscillations in the limit we define the space of $Y$-periodic functions $\rmH^1_\mathrm{per}(Y_*) \subseteq \rmH^1_\mathrm{per}(Y)$ via
\begin{align}
\label{eq:space-per}
\rmH^1_\mathrm{per}(Y_*) := \left\lbrace \Phi \in \rmH^1(Y_*) \,|\, \Phi|_{\Gamma_i} = \Phi|_{\Gamma_{-i}} \right\rbrace ,
\end{align}
where $\Gamma_i$ and $\Gamma_{-i}$ are opposite faces of the unit cube $Y$ with $\partial Y = \bigcup_{i=1}^d \left( \Gamma_i \cup \Gamma_{-i} \right)$.
With this the effective coefficients are given via the standard unit-cell problem
\begin{align}
\label{eq:unit-cell-prob}
\forall\, \xi \in\rr^d :\quad
d_\mathrm{eff}\xi \cdot \xi = \min_{\Phi \in \rmH^1_\mathrm{per}(Y_*)} \int_{Y_*} \mathbb{D} [\nablaY \Phi {+} \xi] \cdot [\nablaY \Phi {+} \xi] \dd y .
\end{align}
Note that the integral is taken over $\int_{Y_*}$ and not the average $\dashint_{Y_*}$. In full, formula \eqref{eq:unit-cell-prob} reads $1/|Y| \int_{Y_*}$ with $|Y|=1$ here. For the boundary data $v_\eps$, we obtain in the limit $\eps\to0$ the usual average
\begin{align}
\label{eq:average}
\forall\, (t,x) \in [0,T]\times\oo:\quad
v_0(t,x) = \dashint_{\partial T} \mathbb{V}(t,x,y) \dd y .
\end{align}

Finally,  we state the existence and uniqueness of solutions for the limit system.
\begin{theorem}
\label{thm:exist-limit}
Let the Assumptions \ref{assump:domain} and \ref{assump:data} hold and let the initial value $(u^0, \Theta^0)$ satisfy 
$u^0 \in \rmH^1(\oo) \cap \rmL^\infty_+(\oo)$ and $\Theta^0 \in \rmL^2(\oo; \rmH^1_\mathrm{per}(Y_*)) \cap \rmL^\infty_+(\oo\times Y_*)$. 
There exists a unique solution $(u,\Theta)$ of \eqref{eq:prob-limit}--\eqref{eq:prob-limit-init} with 
\begin{align}
\label{eq:sol-bound-limit}
\begin{array}{l}
u \in \rmH^1(0,T; \rmL^2(\oo)) \cap \rmL^\infty(0,T; \rmH^1(\oo)) \cap \rmL^\infty_+((0,T)\times \oo) , \\
\Theta \in \rmH^1(0,T; \rmH^1(\oo; \rmL^2(Y_*))) \cap \rmL^\infty(0,T; \rmH^1(\oo; \rmH^1_\mathrm{per}(Y_*))) \cap \rmL^\infty_+((0,T)\times \oo\times Y_*) .
\end{array}
\end{align}
\end{theorem}
\begin{proof}
The existence and boundedness of unique solutions $(u,\Theta)$ follows by Galerkin approximation as in \cite{MNR2010}. In particular, the higher $x$-regularity of $\Theta$ follows by \cite[Thm.~5]{MNR2010}.
\end{proof}

\begin{rem}
By slightly modifying the proof of \cite[Thm.~4]{MNR2010} after equations (40)--(42), 
the assumptions on the initial values can be relaxed from $u^0 \in \rmH^2(\oo)$ and $\Theta^0 \in \rmL^2(\oo; \rmH^2_\mathrm{per}(Y_*))$ to $u^0 \in \rmH^1(\oo)$ and $\Theta^0 \in \rmL^2(\oo; \rmH^1_\mathrm{per}(Y_*))$. To prove the $\rmL^\infty(0,T)$-estimates for the gradients and the $\rmL^2(0,T)$-estimates for the time derivative we can argue as in \cite{KAM2014} by exploiting the symmetry of $d_\mathrm{eff}$ and $\mathbb{K}$ as in \eqref{eq:sym-trick} as well as the fact that the boundary terms are linear.
\end{rem}

\section{Corrector Estimates}

\subsection{Periodic unfolding and folding of two-scale functions}
\label{subsec:two-scale}

The usual two-scale decomposition is given via the mappings $[ \,\cdot\, ]: \rr^d \to \mathbb{Z}^d$ and $\lbrace\, \cdot \,\rbrace: \rr^d \to Y$. For $x \in \rr^d$, $[ x ]$ denotes the component-wise application of the standard Gauss bracket and $\lbrace x \rbrace := x - [x]$ is the remainder. With this, the \emph{periodic unfolding operator} $\T : \rmL^p(\oo_\eps) \to \rmL^p(\oo \times Y_*)$, for $1 \leq p \leq \infty$, is defined by (\cite[Def.~2.3]{CDZ06})
\begin{align}
\label{eq:def-T}
(\T u)(x,y) : = u \left( \eps[\tfrac{x}{\eps}] + \eps y \right) .
\end{align}
Note that we do not need to extend $u$ by $0$ outside $\oo_\eps$ since there occur no intersected cells at the boundary $\partial \oo$, cf.\ also Remark \ref{rem:boundary}. We have indeed $\eps ([\tfrac{x}{\eps}] + y) \in \oo_\eps$ for all $(x,y) \in \oo\times Y_*$ such that $\T$ is well-defined in \eqref{eq:def-T}. 
In the same manner we define the \emph{boundary unfolding operator} $\Tb : \rmL^p(\partial T_\eps) \to \rmL^p(\oo\times\partial T)$ by (\cite[Def.~5.1]{CDZ06})
\begin{align}
\label{eq:def-Tb}
(\Tb u)(x,y) : = u \left( \eps[\tfrac{x}{\eps}] + \eps y \right) .
\end{align}
Following \cite{CDZ06, MT07} we define the \emph{folding (averaging) operator} $\F : \rmL^2(\oo\times Y_*) \to \rmL^2(\oo_\eps)$ via 
\begin{align}
\label{eq:def-F}
(\F U)(x) : = \dashint_{\eps \left( \left[ \tfrac{x}{\eps} \right] + Y_* \right)} U\left( z, \lbrace\tfrac{x}{\eps}\rbrace \right) \dd z \;\bigg\vert_{\oo_\eps},
\end{align}
where $\dashint_A u \dd z = |A|^{-1} \int_A u \dd z$ denotes the usual average and $u|_{\oo_\eps}$ is the restriction of $u$ to $\oo_\eps$.

To derive quantitative estimates for the differences $u_\eps {-} u$ and $\theta_\eps {-} \Theta$, we need to test the weak formulation of the original system with $\rmH^1(\oo_\eps)$-functions which are one-scale pendants of the limiting solution $(u,\Theta)$. There are two options to naively fold a two-scale function $U(x,y)$, namely
\begin{align*}
u_\eps (x) = U(x,\tfrac{x}{\eps}) |_{\oo_\eps}
\quad\text{and}\quad
u^*_\eps (x) = (\F U)(x) .
\end{align*}
However $u_\eps$ is only well-defined in $\rmH^1(\oo_\eps)$, if at least $x \mapsto U(x,y)$ belongs to $\mathrm{C}^1(\oo)$, and our limit $(u,\Theta)$ (respective the corrector $U$ for $\nabla u$) does not satisfy strong differentiability in general. The second option $u^*_\eps$ is neither a suitable test function, since it is not $\rmH^1(\oo_\eps)$-regular. To overcome this regularity issue, we define the \emph{gradient folding operator} following \cite{MT07,Han11,MRT14,Rei2016} and adapt its definition to perforated domains. 

The gradient folding operator $\G : \rmL^2(\oo; \rmH^1_\mathrm{per}(Y_*)) \to \rmH^1(\oo_\eps)$ is defined as follows: for every $U \in \rmL^2(\oo; \rmH^1_\mathrm{per}(Y_*))$, the function $\G U := \wh u_\eps$ is given in $\rmH^1(\oo_\eps)$ as the solution of the elliptic problem
\begin{align}
\label{eq:def-G}
\int_{\oo_\eps} \left( \wh u_\eps - \F U \right) \varphi + (\eps \nabla \wh u_\eps - \F(\nablaY U)) \cdot \eps \nabla \varphi \dd x = 0 
\quad \text{for all } \varphi \in \rmH^1(\oo_\eps) .
\end{align}
Note that $\wh u_\eps$ is uniquely determined by the Lax--Milgram Lemma implying the well-definedness of $\G$. 
For simplicity, we define for $\varphi \in \Hh$ the norm
\begin{align}
\label{eq:norm-ident}
\Vert \varphi \Vert_\eps := \Vert \varphi \Vert_{\rmL^2(\oo_\eps)} + \eps \Vert \nabla \varphi \Vert_{\rmL^2(\oo_\eps)}
\quad\text{with}\quad
\Vert \T\varphi \Vert_{\rmL^2(\oo; \rmH^1(Y_*))} = \Vert \varphi \Vert_\eps ,
\end{align}
where the second identity follows from Lemma \ref{lemma:prop-T}. 
Both folding operators, $\F$ and $\G$, are linear and bounded operators satisfying
\begin{align*}
\Vert \F U \Vert_{\rmL^2(\oo_\eps)} \leq \Vert U\Vert_{\rmL^2(\oo\times Y_*)}
\quad\text{and}\quad
\Vert \G U \Vert_\eps \leq 2\Vert U \Vert_{\rmL^2(\oo; \rmH^1(Y_*))} ,
\end{align*}
where the first estimate is due to Jensen's inequality, while the second one is due to H\"older's inequality.

\subsection{Auxiliary corrector estimates}
\label{subsec:auxiliary}

We are now collecting several results which are essential ingredients in the proof of our error estimates \eqref{main}. 
Note that $u \in \Hh$ also belongs to the space $\rmH^1(\oo_\eps)$ since $\oo_\eps \subset \oo$ and we can apply the unfolding operator via $\T u := \T(\chi_\eps u)$, where $\chi_\eps$ denotes the characteristic function of the set $\oo_\eps$. For the sake of brevity $\chi_\eps$ is omitted in the following.

\begin{lemma}
\label{lemma:unfold-err}
For all $U \in \rmH^1(\oo; \rmL^2(Y_*))$ and $u \in \Hh$ we have
\begin{align*}
\Vert \T\F U {-} U \Vert_{\rmL^2(\oo\times Y_*)} \leq \eps \Vert U \Vert_{\rmH^1(\oo); \rmL^2(Y_*))}
\quad\text{and}\quad
\Vert \T u {-} u \Vert_{\rmL^2(\oo\times Y_*)} \leq \eps \Vert u \Vert_{\Hh} ,
\end{align*}
respectively, where $C>0$ only depends on the domains $\oo$ and $Y_*$.
\end{lemma}
\begin{proof}
The proof for the first estimate is based on the application of the Poincar\'e--Wirtinger inequality on each cell $\eps(\xi + Y_*)$, see \cite[Lem.~3.1]{Rei2016} or \cite[Lem.~2.3.4]{Diss-SR} with $\sigma=0$. The second estimate follows from the first one with 
\begin{align*}
\Vert \T u {-} u \Vert_{\rmL^2(\oo\times Y_*)} \leq
\Vert \T \F u {-} u \Vert_{\rmL^2(\oo\times Y_*)} + 
\Vert \F u {-} u \Vert_{\rmH^1(\oo_\eps)} ,
\end{align*} 
cf.\ also \cite[Eq.~(3.4)]{Gri04}. Note that $\F u$ is indeed well-defined for one-scale functions.
\end{proof}

To control the mollified gradient we prove:
\begin{lemma}
\label{lemma:prop-T-2}
For $\delta > 2\eps\mathrm{diam}(Y)$ and all $u \in \rmL^2(\oo)$ and $(x,y) \in \oo\times Y_*$ we have 
\begin{align}
\label{eq:unfold-nablaD}
\left\vert \T(\nablaD u) - \nablaD u \right\vert (x,y) \leq \sqrt{\eps} C_\delta \Vert u \Vert_{\rmL^2(\oo)} ,
\end{align} 
where $C_\delta > 0$ depends on the mollifier $J_\delta$ and $Y_*$.
\end{lemma}
\begin{proof}
According to \cite[Thm.~6]{Eva98} we obtain for every $(x,y) \in \oo\times Y_*$
\begin{align*}
[\T(\nablaD u)](x,y)  
&= \left[\T \left( \int_{B(x,\delta)} \nabla J_\delta(x - \xi)u(\xi) \dd\xi \right) \right] (x,y) \\
%& = \frac{1}{\eps} \nabla_y \left( \T  \int_{B(x,\delta)} J_\delta(x - \xi)u(\xi) \dd\xi \right)
%= \T \left( \nabla \int_{B(x,\delta)} J_\delta(x - \xi)u(\xi) \dd\xi \right) \\
%& = \frac{1}{\eps} \nabla_y \left( \T  \int_{B(x,\delta)} J_\delta(x - \xi)u(\xi) \dd\xi \right) 
%= \frac{1}{\eps} \int_{B \left( \eps \left[\tfrac{x}{\eps} \right] + \eps y,\delta \right)} \nablaY J_\delta(\eps[\tfrac{x}{\eps}] {+} \eps y - \xi)u(\xi) \dd\xi \\
& = \int_{B \left( \eps \left[\tfrac{x}{\eps} \right] + \eps y,\delta \right)} \nabla J_\delta(\eps[\tfrac{x}{\eps}] {+} \eps y - \xi)u(\xi) \dd\xi. 
\end{align*}
For $\delta > 2\eps\mathrm{diam}(Y)$, we define the following $d$-dimensional annulus 
\begin{align*}
B_\mathrm{diff} := B(x, \delta {+} \eps \mathrm{diam}(Y)) \backslash B(x, \delta {-} \eps \mathrm{diam}(Y))
\end{align*}
of thickness $\eps 2\mathrm{diam}(Y)$ and with volume $| B_\mathrm{diff} | \leq \eps \mathrm{Const}(\delta,Y)$. We arrive at 
\begin{align*}
&\left\vert \T(\nablaD u) - \nablaD u \right\vert (x,y)  \\
& = \left\vert \int_{B \left( \eps \left[\tfrac{x}{\eps} \right] + \eps y,\delta \right)} \nabla J_\delta(\eps[\tfrac{x}{\eps}] {+} \eps y - \xi)u(\xi) \dd\xi 
- \int_{B(x,\delta)} \nabla J_\delta(x - \xi)u(\xi) \dd\xi \right\vert \\
& \leq \left\vert \int_{B_\mathrm{diff}} \nabla J_\delta(x - \xi)u(\xi) \dd\xi \right\vert 
\leq \Vert \nabla J_\delta \Vert_{\rmL^2(B_\mathrm{diff})} \Vert u\Vert_{\rmL^2(B_\mathrm{diff})} 
\leq \sqrt{\eps} C \Vert J_\delta \Vert_{\mathrm{C}^\infty(\rr^d)} \Vert u\Vert_{\rmL^2(\oo)}, 
\end{align*}
which proves the assertion. 
\end{proof}

Having defined two folding operators, $\F$ being dual to $\T$ and $\G$ assuring $\rmH^1$-regularity, we call  their difference \emph{folding mismatch} and control it as follows.

\begin{theorem}[Folding mismatch]
\label{thm:fold-mismatch}
For $C>0$ only depending on $\oo$ and $Y_*$ it holds
\begin{align}
%&\Vert \G(u,U) - u|_{\oo_\eps} \Vert_{\rmL^2(\oo_\eps)} + \Vert \nabla \G(u,U) - \F[\nabla u {+} \nablaY U ] \Vert_{\rmL^2(\oo_\eps)} \nonumber \\
%\label{eq:fold-est}
%& \hspace{225pt} \leq \eps C(\Vert u \Vert_{\rmH^1(\oo)} +\Vert U \Vert_{\rmH^1(\oo;\rmH^1(Y_*))}), \\
\label{eq:fold-est}
&\Vert \G U - \F U \Vert_{\rmL^2(\oo_\eps)} + \Vert \eps \nabla \G U - \F(\nablaY U ) \Vert_{\rmL^2(\oo_\eps)} \leq \eps C \Vert U \Vert_{\rmH^1(\oo;\rmH^1(Y_*))} . 
\end{align}
\end{theorem}
\begin{proof}
The proof is based on \cite[Sec.~3.2]{Rei2016} and adapted to perforated domains in Appendix \ref{app:AppendixB}.
\end{proof}

Since unfolded Sobolev functions $\T u \in \rmL^2(\oo; \rmH^1(Y_*)) \supsetneqq \rmL^2(\oo; \rmH^1_\mathrm{per}(Y_*))$ are in general not $Y$-periodic, we need to control the so-called \emph{periodicity defect}, cf.\ \cite{Gri04,Gris05}. In the case of slow diffusion it reads: 
\begin{theorem}[Periodicity defect I]
\label{thm:per-defect}
For every $\varphi \in \rmH^1(\oo_\eps)$, there exists a $Y$-periodic function $\Phi_\eps \in \rmL^2(\oo; \rmH^1_\mathrm{per}(Y_*))$ such that
\begin{align*}
\Vert \Phi_\eps \Vert_{\rmH^1(Y_*; \rmL^2(\oo))} \leq C \Vert \varphi \Vert_\eps 
\quad\text{and}\quad
\Vert \T  \varphi - \Phi_\eps \Vert_{\rmH^1(Y_*; \Hh^*)} \leq \sqrt{\eps} C \Vert \varphi \Vert_\eps  ,
\end{align*}
where the constant $C>0$ only depends in the domains $\oo$ and $Y_*$.
\end{theorem}
\begin{proof}
The proof relies on \cite[Thm.\,2.2]{Gris05} which we can apply after suitably extending $\varphi$ from the perforated domain $\oo_\eps$ to the whole domain $\oo$. 
Let $\wt \varphi \in \Hh$ denote the extension of $\varphi$ as in Theorem \ref{thm:extension}. According to \cite[Thm.\,2.2]{Gris05}, there exists a two-scale function $\wh \Phi_\eps \in \rmL^2(\oo; \rmH^1_\mathrm{per}(Y))$ satisfying
\begin{align*}
&\Vert \wh \Phi_\eps \Vert_{\rmH^1(Y; \rmL^2(\oo))} \leq C \Vert \wt\varphi \Vert_{\eps,\oo} 
\quad\text{and}\quad
\Vert \T^\mathrm{G}  \wt \varphi - \wh \Phi_\eps \Vert_{\rmH^1(Y; \Hh^*)} \leq \sqrt{\eps} C \Vert \wt \varphi \Vert_{\eps,\oo} , \\
&\text{where}\quad
\Vert w \Vert_{\eps,\oo} := \Vert w \Vert_{\rmL^2(\oo)} + \eps \Vert \nabla w \Vert_{\rmL^2(\oo)} ,
\end{align*}
with $C>0$ only depending on $\oo$, $Y$, and $\T^\mathrm{G} : \rmL^2(\oo) \to \rmL^2(\oo\times Y)$ defined on the whole unit-cell as in \eqref{eq:def-T}, cf.\ also \cite{Gri04,Gris05}. Note that it holds $\T \varphi = (\T^\mathrm{G} \wt \varphi) |_{\oo\times Y_*}$. Recalling the definition of $\rmH^1_\mathrm{per}(Y_*)$ in \eqref{eq:space-per} with $Y_* \subseteq Y$, we define $\Phi_\eps \in \rmL^2(\oo; \rmH^1_\mathrm{per}(Y_*))$ via $\Phi_\eps := \wh \Phi_\eps|_{\oo\times Y_*}$, which gives
\begin{align*}
\Vert \Phi_\eps \Vert_{\rmH^1(Y_*; \rmL^2(\oo))} 
\leq \Vert \wh \Phi_\eps \Vert_{\rmH^1(Y; \rmL^2(\oo))} 
,\quad
\Vert \T \varphi -  \Phi_\eps \Vert_{\rmH^1(Y_*; \Hh^*)} 
\leq \Vert \T^\mathrm{G}  \wt \varphi - \wh \Phi_\eps \Vert_{\rmH^1(Y; \Hh^*)} 
\end{align*}
and the proof is finished. 
\end{proof}

For the case of classical diffusion, we consider $\varphi \in \Hh$ instead of $\rmH^1(\oo_\eps)$. This is related to the fact that, in the limit system, the $u$-equation is given in the macroscopic domain $\oo$, whereas the $\Theta$-equation as posed in the two-scale space $\oo\times Y_*$, and hence, it cannot be reduced to $\oo$ only.

\begin{theorem}[Periodicity defect II]
\label{thm:per-defect-2}
For every $\varphi \in \Hh$, there exists a $Y$-periodic function $\Phi_\eps \in \rmL^2(\oo; \rmH^1_\mathrm{per}(Y_*))$ such that
\begin{align*}
\Vert \Phi_\eps \Vert_{\rmH^1(Y_*; \rmL^2(\oo))} \leq C \Vert \varphi \Vert_{\Hh}
\quad\text{and}\quad
\Vert \nabla \varphi {+} \nablaY \Phi_\eps  - \T(\nabla \varphi) \Vert_{\rmL^2(Y_*; \Hh^*)} \leq \sqrt{\eps} C \Vert \varphi \Vert_{\Hh}  ,
\end{align*}
where the constant $C>0$ only depends in the domains $\oo$ and $Y_*$.
\end{theorem}
\begin{proof}
For $\varphi \in \Hh$ the desired estimates hold with $\wh\Phi_\eps \in \rmL^2(\oo; \rmH^1_\mathrm{per}(Y))$ according to \cite[Thm.\,2.3]{Gris05}. Choosing $\Phi_\eps = \wh \Phi_\eps|_{\oo\times Y_*}$ as in the proof of Theorem \ref{thm:per-defect} yields the assertion.
\end{proof}

\subsection{Main Theorem and its proof}
\label{subsec:proof}

Having collected all preliminaries, we can now state and prove the corrector estimates for our thermo-diffusion model. 
\begin{theorem}
\label{thm:main}
Let $(u_\eps,\theta_\eps)$ and $(u,\Theta)$ denote the unique solution of $(\mathrm{P}_\eps)$ and $(\mathrm{P}_0)$, respectively, according to Theorem \ref{thm:exist-eps} and  Theorem \ref{thm:exist-limit}. 
If the initial values satisfy
\begin{align}
\label{eq:initial-est}
\exists\, C_0> 0:\quad \Vert \T u_\eps^0 {-} u^0 \Vert_{\rmL^2(\oo\times Y_*)}
+\Vert \T \theta_\eps^0 {-} \Theta^0 \Vert_{\rmL^2(\oo\times Y_*)}
\leq \sqrt{\eps} C_0, 
\end{align}
then we have 
\begin{align}
\label{eq:error-est}
&\Vert \T u_\eps {-} u \Vert_{\rmL^\infty(0,T; \rmL^2(\oo\times Y_*))}
+ \Vert \T(\nabla u_\eps) {-} (\nabla u {+} \nablaY U) \Vert_{\rmL^2(0,T; \rmL^2(\oo\times Y_*))} \nonumber \\
& + \Vert \T \theta_\eps {-} \Theta \Vert_{\rmL^\infty(0,T; \rmL^2(\oo\times Y_*))}
+ \Vert \T(\eps\nabla \theta_\eps) {-} \nablaY \Theta \Vert_{\rmL^2(0,T; \rmL^2(\oo\times Y_*))}
\leq \sqrt{\eps}C,
\end{align}
where the constant $C>0$ depends on the given data and the norms in \eqref{eq:sol-bound} and \eqref{eq:sol-bound-limit}. 
\end{theorem}

\begin{proof}
Note that the domain $\oo$ is convex, bounded, and has a Lipschitz boundary. Since $\dot{u}$ and $v_0$ belong to the space $\rmL^2((0,T)\times \oo)$, we can apply \cite[Thm.~3.2.1.3]{Grisvard1985} and obtain that the limit $u(t,\cdot)$ belongs to the better space $\rmH^2(\oo)$. 

If not stated otherwise, the following notion of weak formulation is to be understood pointwise in $[0,T]$.

\noindent\textbf{Part A: Slow diffusion.} 
Note that for $u_\eps \in \rmH^1(\oo_\eps)$ and $u\in\Hh$ the following two norms are equivalent up to an error of order $O(\eps)$, i.e.\
\begin{align}
\label{eq:norm-diff-eps}
\left\vert \Vert \T u_\eps {-} u \Vert_{\rmL^2(\oo\times Y_*)} 
- \Vert u_\eps {-} u \Vert_{\rmL^2(\oo_\eps)} \right\vert
\leq \eps C \Vert u\Vert_{\Hh} ,
\end{align}
which is due to $\Vert \T u {-} u \Vert_{\rmL^2(\oo\times Y_*)} \leq \eps C \Vert u\Vert_{\Hh}$ by Lemma \ref{lemma:unfold-err}.

\emph{Step 1: Reformulation of $\theta_\eps$-equation.} The weak formulation of the $\theta_\eps$-equation reads 
\begin{align}
\int_{\oo_\eps} \dot{\theta}_\eps \psi \dd x = \int_{\oo_\eps} {-} \kappa_\eps \eps\nabla \theta_\eps \cdot \eps\nabla \psi + \mu \eps \nabla \theta_\eps \cdot \nablaD u_\eps \psi \dd x 
+ \int_{\partial T_\eps} \eps g \theta_\eps \psi \dd\sigma 
\end{align}
for all admissible test functions $\psi \in \rmH^1(\oo_\eps)$. Applying the periodic unfolding operators $\T$ and $\Tb$, with $\T \kappa_\eps = \mathbb{K}$, and exploiting their properties in Lemma \ref{lemma:prop-T} and \ref{lemma:prop-T-2} gives 
\begin{align}
\label{eq:slow-eps-1}
\int_{\oo \times Y_*} \T\dot{\theta}_\eps \T\psi \dd x \dd y 
&= \int_{\oo\times Y_*} {-} \mathbb{K} \nablaY (\T \theta_\eps) \cdot \nablaY (\T\psi) 
+ \mu \nablaY (\T \theta_\eps) \cdot \T(\nablaD u_\eps) \T\psi \dd x \nonumber \\
&\quad + \int_{\oo\times\partial T} g \Tb\theta_\eps \Tb\psi \dd x\dd\sigma(y)  .
\end{align}

We choose $\psi := \theta_\eps {-} \G \Theta$ in \eqref{eq:slow-eps-1}, which is by construction of the gradient folding operator $\G$ an admissible test function in $\rmH^1(\oo_\eps)$  so that
\begin{align}
\label{eq:slow-eps-2}
&\int_{\oo \times Y_*} \T\dot{\theta}_\eps \T(\theta_\eps {-} \G \Theta) \dd x \dd y \nonumber \\
& = \int_{\oo\times Y_*} {-} \mathbb{K} \nablaY (\T \theta_\eps) \cdot \nablaY [\T (\theta_\eps {-} \G \Theta)] 
+ \mu \nablaY (\T\theta_\eps) \cdot \T(\nablaD u_\eps) \T(\theta_\eps {-} \G\Theta) \dd x \dd y \nonumber \\
&\quad + \int_{\oo\times\partial T} g \Tb\theta_\eps \Tb(\theta_\eps {-} \G \Theta) \dd x\dd\sigma(y) .
\end{align}
Adding $\pm \Theta$ respective $\pm \nablaY\Theta$ gives 
\begin{align}
\label{eq:slow-eps-3}
&\int_{\oo \times Y_*} \T\dot{\theta}_\eps (\T\theta_\eps {-} \Theta) \dd x \dd y \nonumber \\
& = \int_{\oo\times Y_*} {-} \mathbb{K} \nablaY (\T \theta_\eps) \cdot \nablaY (\T \theta_\eps {-}  \Theta) 
+ \mu \nablaY (\T\theta_\eps) \cdot \T(\nablaD u_\eps) (\T\theta_\eps {-} \Theta) \dd x \dd y \nonumber \\
&\quad + \int_{\oo\times\partial T} g \Tb\theta_\eps (\Tb\theta_\eps {-} \Theta) \dd x\dd\sigma(y)
+ \Delta^{\theta_\eps}_\mathrm{fold} ,
\end{align}
where the folding mismatch $\Delta^{\theta_\eps}_\mathrm{fold}$ reads
\begin{align}
\Delta^{\theta_\eps}_\mathrm{fold} 
&:= \int_{\oo\times Y_*} \big\lbrace \T\dot{\theta}_\eps (\T\G\Theta {-} \Theta)
{-} \mathbb{K} \nablaY (\T \theta_\eps) \cdot \nablaY ( \Theta {-} \T\G \Theta) \nonumber \\
& \hspace{8em}  + \mu \nablaY (\T\theta_\eps) \cdot \T(\nablaD u_\eps) (\Theta {-} \T\G\Theta) \big\rbrace \dd x \dd y \nonumber \\
&\quad + \int_{\oo\times\partial T} g \Tb\theta_\eps (\Theta {-} \Tb\G \Theta) \dd x\dd\sigma(y) .
\end{align}
To treat the boundary term,  we exploit the continuous embedding $\rmL^2(\oo; \rmH^1(Y_*)) \subset \rmL^2(\oo\times\partial T)$, i.e.\ 
\begin{align}
\label{eq:emb}
\exists\, C_\mathrm{emb} > 0 , \forall\, \Psi \in \rmL^2(\oo; \rmH^1(Y_*)): \quad
\Vert \Psi \Vert_{\rmL^2(\oo\times\partial T)} 
\leq C_\mathrm{emb} \Vert \Psi \Vert_{\rmL^2(\oo; \rmH^1(Y_*))} .
\end{align}
Using the $\nablaD$-estimates in \eqref{eq:est-delta} 
as well as the boundedness of the solution $(u_\eps,\theta_\eps)$ in \eqref{eq:sol-bound}, 
in particular, the improved time-regularity $\Vert \dot{\theta}_\eps \Vert_{\rmL^2((0,T)\times\oo_\eps)} < \infty$  and $\Vert \T \theta_\eps \Vert_{\rmL^2(\oo; \rmH^1(Y_*))} 
= \Vert \theta_\eps \Vert_\eps$ by \eqref{eq:norm-ident}, gives
\begin{align*}
\int_0^T |\Delta^{\theta_\eps}_\mathrm{fold} | \dd t 
\leq C \Vert \T\G \Theta {-} \Theta \Vert_{\rmL^2((0,T)\times\oo; \rmH^1(Y_*))} . 
\end{align*}
Inserting $\pm \F \Theta$ respective $\pm \F(\nablaY \Theta)$, applying the triangle inequality, and using the norm preservation of $\T$ gives 
\begin{align*}
\Vert \T\G \Theta {-} \Theta \Vert_{\rmL^2(\oo; \rmH^1(Y_*))} 
& \leq \Vert \G \Theta {-} \F\Theta\Vert_{\rmL^2(\oo_\eps)} + \Vert \eps \nabla (\G \Theta) {-} \F(\nabla_y\Theta)\Vert_{\rmL^2(\oo_\eps)} \\
& \quad + \Vert \T\F\Theta {-} \Theta\Vert_{\rmL^2(\oo\times Y_*)} + \Vert \T\F(\nablaY\Theta) {-} \nablaY\Theta\Vert_{\rmL^2(\oo\times Y_*)} ..
\end{align*}
Using the higher $x$-regularity $\Theta \in \rmL^2(0,T; \rmH^1(\oo; \rmH^1_\mathrm{per}(Y_*)))$, applying Proposition \ref{thm:fold-mismatch} for the folding mismatch, and Lemma \ref{lemma:unfold-err}  for the unfolding error gives 
\begin{align}
\int_0^T |\Delta^{\theta_\eps}_\mathrm{fold} | \dd t \leq O(\eps). 
\end{align}

\emph{Step 3: Reformulation of $\Theta$-equation.} The weak formulation of the $\Theta$-equation reads 
\begin{align}
\int_{\oo\times Y_*} \dot{\Theta} \Psi \dd x\dd y 
= \int_{\oo\times Y_*} {-} \mathbb{K} \nablaY \Theta \cdot \nablaY \Psi + \mu \nablaY \Theta \cdot \nablaD u \Psi \dd x\dd y 
+ \int_{\oo\times\partial T} g \Theta \Psi \dd x \dd\sigma(y)
\end{align}
for all admissible test functions $\Psi \in \rmL^2(\oo; \rmH^1_\mathrm{per}(Y_*))$. We choose $\Psi_\eps$ according to Proposition \ref{thm:per-defect} such that we can control the periodicity defect of $\T \psi$ for arbitrary functions $\psi \in \rmH^1(\oo_\eps)$, namely
\begin{align}
\label{eq:slow-limit-1}
\int_{\oo\times Y_*} \dot{\Theta} \T \psi \dd x\dd y 
&= \int_{\oo\times Y_*} {-} \mathbb{K} \nablaY \Theta \cdot \nablaY\T\psi 
+ \mu \nablaY \Theta \cdot \nablaD u \T\psi \dd x\dd y \nonumber \\
&\quad + \int_{\oo\times\partial T} g \Theta \Tb\psi \dd x \dd\sigma(y) 
+ \Delta^{\Theta}_\mathrm{per},
\end{align}
where the periodicity defect $\Delta^{\Theta}_\mathrm{per}$ is given via 
\begin{align}
\label{eq:slow-delta-per}
\Delta^{\Theta}_\mathrm{per} 
& := \int_{\oo\times Y_*} \dot{\Theta} (\T \psi {-} \Psi_\eps)
- \mathbb{K} \nablaY \Theta \cdot \nablaY (\Psi {-} \T \psi) 
+ \mu \nablaY \Theta \cdot \nablaD u (\Psi_\eps {-} \T\psi) \dd x\dd y \nonumber \\
&\quad  + \int_{\oo\times\partial T} g \Theta (\Psi_\eps {-} \Tb\psi) \dd x \dd\sigma(y).
\end{align}
Applying H\"older's inequality and the embedding \eqref{eq:emb} yield
\begin{align*}
\left\vert \Delta^{\Theta}_\mathrm{per} \right\vert
& \leq (1 + C_\mathrm{emb}) \big\lbrace  
\Vert \dot{\Theta} \Vert_{\rmL^2(Y_*; \Hh)}
+ C_\mathrm{elip}^{-1} \Vert \nablaY \Theta \Vert_{\rmL^2(Y_*; \Hh)}
+ \mu \Vert \nablaY \Theta \cdot \nablaD u \Vert_{\rmL^2(Y_*; \Hh)} \\
&\hspace{200pt} + g \Vert \Theta \Vert_{\rmL^2(\partial T; \Hh)} \big\rbrace
\Vert \T \psi {-} \Psi \Vert_{\rmH^1(Y_*; \Hh^*)} .
\end{align*}
According to the higher $x$-regularity of $\Theta$ in \eqref{eq:sol-bound-limit} it is $\nablaY \Theta \cdot \nablaD u \in \rmL^2(Y_*; \Hh)$ with $\nablaD u \in \mathrm{W}^{1,\infty}(\oo)$ such that we can apply Theorem \ref{thm:per-defect} and obtain 
\begin{align}
\label{eq:slow-limit-2}
\int_0^T |\Delta^{\Theta}_\mathrm{per} | \dd t
\leq C \Vert \T \psi {-} \Psi_\eps \Vert_{\rmH^1(Y_*; \rmH^1(\oo)^*)}
\leq O(\sqrt{\eps}) \Vert \T\psi \Vert_{\rmL^2((0,T)\times\oo; \rmH^1(Y_*))} .
\end{align}

Choosing $\psi := \theta_\eps {-} \G \Theta$ in  \eqref{eq:slow-limit-1} yields
\begin{align}
\label{eq:slow-limit-2}
\int_{\oo\times Y_*} \dot{\Theta} \T (\theta_\eps {-} \G \Theta) \dd x\dd y 
&= \int_{\oo\times Y_*} \big\lbrace {-} \mathbb{K} \nablaY \Theta \cdot \nablaY \T(\theta_\eps {-} \G \Theta) \nonumber \\
&\hspace{50pt} + \mu \nablaY \Theta \cdot \nablaD u \T(\theta_\eps {-} \G \Theta) \big\rbrace \dd x\dd y \nonumber \\
&\quad+ \int_{\oo\times\partial T} g \Theta \Tb(\theta_\eps {-} \G \Theta) \dd x \dd\sigma(y) 
+ \Delta^{\theta_\eps}_\mathrm{per} .
\end{align}
Adding $\pm \Theta$, and respectively $\pm \nablaY\theta$, as in Step 1 gives
\begin{align}
\label{eq:slow-limit-3}
\int_{\oo\times Y_*} \dot{\Theta} (\T\theta_\eps {-} \Theta) \dd x\dd y 
&= \int_{\oo\times Y_*} {-} \mathbb{K} \nablaY \Theta \cdot \nablaY (\T\theta_\eps {-} \Theta)
+ \mu \nablaY \Theta \cdot \nablaD u (\T\theta_\eps {-} \Theta) \dd x\dd y \nonumber \\
&\quad + \int_{\oo\times\partial T} g \Theta (\Tb \theta_\eps {-} \Theta) \dd x \dd\sigma(y) 
+ \Delta^{\Theta}_\mathrm{per,fold},
\end{align}
where the folding mismatch $\Delta^{\Theta}_\mathrm{fold}$ is determined by 
\begin{align}
\label{eq:slow-limit-delta-fold}
\Delta^{\Theta}_\mathrm{fold} 
& := \int_{\oo\times Y_*} \dot{\Theta} (\T\G \Theta {-} \Theta)
- \mathbb{K} \nablaY \Theta \cdot \nablaY (\Theta {-} \T \G \Theta) 
+ \mu \nablaY \Theta \cdot \nablaD u (\Theta {-}\T\G\Theta) \dd x\dd y \nonumber \\
&\quad + \int_{\oo\times\partial T} g \Theta (\Theta {-} \Tb\G\Theta) \dd x \dd\sigma(y)  .
\end{align}
The estimation of $\Delta^{\Theta}_\mathrm{fold}$ follows along the lines of $\Delta^{\theta_\eps}_\mathrm{fold}$ in Step 1 using the boundedness of the limit $(u,\Theta)$, in particular, the boundedness of $\Vert \partial_t \Theta \Vert_{\rmL^2((0,T)\times\oo\times Y_*)}$. 
%Hence it is $|\Delta^{\Theta}_\mathrm{fold}| \leq O(\eps)$. 
Finally,  we insert the test function $\psi := \theta_\eps {-} \G \Theta$ into the $\Delta^\Theta_\mathrm{per}$-estimate in \eqref{eq:slow-limit-2} and apply Young's inequality with $\eta_1 > 0$ 
\begin{align}
\label{eq:slow-gronwall-0}
\int_0^T |\Delta^{\theta_\eps}_\mathrm{fold}| + |\Delta^{\Theta}_\mathrm{per}| + |\Delta^{\Theta}_\mathrm{fold}| \dd t
\leq \eps C_{\eta_1} + \eta_1 \Vert \T\theta_\eps {-} \Theta \Vert^2_{\rmL^2((0,T)\times \oo; \rmH^1(Y_*))} .
\end{align}

\emph{Step 4: Derivation of Gr\"onwall-type estimates.}
Subtracting equation \eqref{eq:slow-limit-3} from \eqref{eq:slow-eps-3} and using $\tfrac12 \tfrac{\dd}{\dd t} \Vert \Psi \Vert^2_{\rmL^2(\oo\times Y_*)} = \int_{\oo\times Y_*} \dot{\Psi} \Psi \dd x\dd y$ gives
\begin{align}
\label{eq:slow-gronwall}
\frac12 \frac{\dd}{\dd t} \Vert \T \theta_\eps {-} \Theta \Vert^2_{\rmL^2(\oo\times Y_*)} 
&= \int_{\oo\times Y_*} \big\lbrace {-} \mathbb{K} [\nablaY (\T\theta_\eps {-} \Theta)] \cdot [\nablaY (\T\theta_\eps {-} \Theta)] \nonumber \\
& \hspace{50pt} + \mu [\nablaY(\T\theta_\eps) \cdot \T(\nablaD u_\eps) - \nablaY\Theta\cdot\nablaD u](\T\theta_\eps {-} \Theta) \big\rbrace \dd x\dd y \nonumber \\
& \quad + \int_{\oo\times\partial T}  g |\Tb\theta_\eps {-} \Theta |^2 \dd x\dd\sigma(y) \nonumber \\
& \quad + \Delta^{\theta_\eps}_\mathrm{fold} - \Delta^{\Theta}_\mathrm{per,fold} .
\end{align}
We continue by estimating each term on the right-hand side in \eqref{eq:slow-gronwall-1} separately. 
Exploiting the interpolation inequality (cf. e.g. \cite{LiM72})
\begin{align*}
\exists\, C_\mathrm{int} > 0, \forall\, \Psi \in \rmL^2(\oo; \rmH^1(Y_*)):\quad 
\Vert \Psi \Vert^2_{\rmL^2(\oo\times\partial T)} 
& \leq C_\mathrm{int} \Vert \Psi \Vert_{\rmL^2(\oo\times Y_*)} 
\Vert \Psi \Vert_{\rmL^2(\oo; \rmH^1(Y_*))}
\end{align*}
and then Young's inequality with $\eta_2 > 0$ lead to  
\begin{align}
\label{eq:slow-gronwall-1}
\Vert \Tb\theta_\eps {-} \Theta\Vert^2_{\rmL^2(\oo\times\partial T)} 
%& \leq C_\mathrm{int} \Vert \T\theta_\eps {-} \Theta \Vert_{\rmL^2(\oo\times Y_*)} 
%\Vert \T\theta_\eps {-} \Theta \Vert_{\rmL^2(\oo; \rmH^1(Y_*))} \nonumber \\
& \leq C_{\eta_2} \Vert \T\theta_\eps {-} \Theta \Vert^2_{\rmL^2(\oo\times Y_*)} 
+ \eta_2 \Vert \T\theta_\eps {-} \Theta \Vert^2_{\rmL^2(\oo; \rmH^1(Y_*))} .
\end{align}
Reformulating the $\mu$-term gives
\begin{align}
&\int_{\oo\times Y_*} \mu [\nablaY(\T\theta_\eps) \cdot \T(\nablaD u_\eps) - \nablaY\Theta\cdot\nablaD u](\T\theta_\eps {-} \Theta) \dd x\dd y \nonumber \\
\label{eq:slow-gronwall-mu-1}
&= \int_{\oo\times Y_*} \mu \nablaY(\T\theta_\eps) \cdot [\T(\nablaD u_\eps) {-} \nablaD u](\T\theta_\eps {-} \Theta)  \dd x\dd y \\
\label{eq:slow-gronwall-mu-2}
&\quad + \int_{\oo\times Y_*} \mu \nablaY(\T\theta_\eps {-} \Theta) \cdot\nablaD u(\T\theta_\eps {-} \Theta)  \dd x\dd y .
\end{align}
Using in \eqref{eq:slow-gronwall-mu-2} that $\Vert \nablaD u\Vert_{\rmL^\infty(\oo\times Y_*)} = \Vert \nablaD u\Vert_{\rmL^\infty(\oo)} \leq C_\delta \Vert u \Vert_{\rmL^2(\oo)}$ is bounded as well as H\"older's and Young's inequality with $\eta_3 > 0$ gives
\begin{align}
\label{eq:slow-gronwall-2}
&\left\vert \int_{\oo\times Y_*} \mu \nablaY(\T\theta_\eps {-} \Theta) \cdot\nablaD u(\T\theta_\eps {-} \Theta)  \dd x\dd y \right\vert \nonumber\\
& \leq \mu C_\delta \Vert u \Vert_{\rmL^2(\oo)} 
\Vert \nablaY(\T\theta_\eps {-} \Theta) \Vert_{\rmL^2(\oo\times Y_*)} 
\Vert \T\theta_\eps {-} \Theta \Vert_{\rmL^2(\oo\times Y_*)} \nonumber\\
&\leq C_{\eta_3} \Vert \T\theta_\eps {-} \Theta \Vert^2_{\rmL^2(\oo\times Y_*)} 
+ \eta_3 \Vert \nablaY(\T\theta_\eps {-} \Theta) \Vert^2_{\rmL^2(\oo\times Y_*)} .
\end{align}
In a similar manner we obtain for \eqref{eq:slow-gronwall-mu-1} by adding $\pm \F(\nablaD u)$ and using estimate \eqref{eq:unfold-nablaD}
\begin{align}
\label{eq:slow-gronwall-3}
&\left\vert \int_{\oo\times Y_*} \mu \nablaY(\T\theta_\eps) \cdot [\T(\nablaD u_\eps) {-} \nablaD u](\T\theta_\eps {-} \Theta) \dd x\dd y \right\vert \nonumber\\
&\leq \mu \Vert \nablaY(\T\theta_\eps) \Vert_{\rmL^2(\oo\times Y_*)}
\Vert \T(\nablaD u_\eps) {-} \nablaD u \Vert_{\rmL^\infty(\oo\times Y_*)}
\Vert \T\theta_\eps {-} \Theta \Vert_{\rmL^2(\oo\times Y_*)} \nonumber\\
& \leq 
C \Vert \T \theta_\eps {-} \Theta \Vert_{\rmL^2(\oo\times Y_*)} 
\left\lbrace \Vert \T (\nablaD u_\eps) {-} \T(\nablaD u) \Vert_{\rmL^\infty(\oo\times Y_*)} + \Vert \T (\nablaD u) {-} \nablaD u \Vert_{\rmL^\infty(\oo\times Y_*)} 
 \right\rbrace \nonumber\\
& \leq C_\delta \Vert \T \theta_\eps {-} \Theta \Vert_{\rmL^2(\oo\times Y_*)} 
\left\lbrace \Vert u_\eps {-} u \Vert_{\rmL^2(\oo_\eps)}
+ \sqrt{\eps} \Vert u \Vert_{\rmL^2(\oo)} \right\rbrace \nonumber\\
& \leq C \left\lbrace
\Vert \T \theta_\eps {-} \Theta \Vert^2_{\rmL^2(\oo\times Y_*)} 
+ \Vert u_\eps {-} u \Vert^2_{\rmL^2(\oo_\eps)}
+ \eps \Vert u \Vert^2_{\rmL^2(\oo)} \right\rbrace .
\end{align}
Overall we can estimate equation \eqref{eq:slow-gronwall} with the uniform ellipticity of $\mathbb{K}$ and \eqref{eq:slow-gronwall-1}--\eqref{eq:slow-gronwall-3} such that
\begin{align*}
\frac12\frac{\dd}{\dd t} \Vert \T \theta_\eps {-} \Theta \Vert^2_{\rmL^2(\oo\times Y_*)} 
&\leq (\eta_1 + \eta_2 + \eta_3 {-} C_\mathrm{elip} ) \Vert \nablaY (\T\theta_\eps {-} \Theta) \Vert^2_{\rmL^2(\oo\times Y_*)}  \\
&\quad + C \left\lbrace \Vert \T \theta_\eps {-} \Theta \Vert^2_{\rmL^2(\oo\times Y_*)}
+ \Vert u_\eps {-} u \Vert^2_{\rmL^2(\oo_\eps)} \right\rbrace \\
&\quad + |\Delta^{\theta_\eps}_\mathrm{fold}| + |\Delta^\Theta_\mathrm{per}| + |\Delta^\Theta_\mathrm{fold}|
+ \eps \Vert u \Vert^2_{\rmL^2(\oo)} .
\end{align*}
Choosing $\eta_i = C_\mathrm{elip}/6$, integrating over $[0,t]$ with $0<t\leq T$, as well as recalling \eqref{eq:norm-diff-eps} and \eqref{eq:slow-gronwall-0} yields 
\begin{align}
\label{eq:slow-gronwall-final}
&\Vert \T \theta_\eps(t) {-} \Theta(t) \Vert^2_{\rmL^2(\oo\times Y_*)} 
+ C_\mathrm{elip} \Vert \nablaY (\T\theta_\eps {-} \Theta) \Vert^2_{\rmL^2((0,t)\times \oo\times Y_*)} \nonumber \\
& \leq C \left\lbrace \Vert \T \theta_\eps {-} \Theta \Vert^2_{\rmL^2((0,t)\times\oo\times Y_*)}
+ \Vert \T u_\eps {-} u \Vert^2_{\rmL^2((0,t)\times\oo\times Y_*)} \right\rbrace \nonumber \\
&\quad + \Vert \T \theta_\eps^0 {-} \Theta^0 \Vert_{\rmL^2(\oo\times Y_*)} + O(\eps) .
\end{align}

\noindent\textbf{Part B: Classical diffusion.}
We point out that the higher regularity of the limit $u \in \rmH^2(\oo)$ implies the higher $x$-regularity of the corrector $U \in \rmH^1(\oo; \rmH^1_\mathrm{per}(Y_*))$ which is the unique minimizer of the unit-cell problem \eqref{eq:unit-cell-prob} with $\xi = \nabla u(x)$. 

\emph{Step 1: Reformulation of $u_\eps$-equation.} The weak formulation of the $u_\eps$-equation is given via
\begin{align*}
\int_{\oo_\eps} \dot{u}_\eps \varphi \dd x 
= \int_{\oo_\eps} - d_\eps \nabla u_\eps \cdot \nabla \varphi + \eps \tau \nabla u_\eps \cdot \nablaD \theta_\eps \varphi + R(u_\eps) \varphi \dd x
+ \int_{\partial T_\eps} \eps (a u_\eps + b v_\eps)\varphi \dd\sigma  
\end{align*}
for all test functions $\varphi \in \rmH^1(\oo_\eps)$. First of all note that the cross-diffusion term 
\begin{align*}
\Delta^{u_\eps}_\mathrm{cross} := \int_{\oo_\eps} \eps \tau \nabla u_\eps \cdot \nablaD \theta_\eps \varphi\dd x
\end{align*}
is of order $O(\eps)$ thanks to H\"older's inequality and the boundedness in \eqref{eq:sol-bound} and \eqref{eq:est-delta}
\begin{align}
\label{eq:class-eps-1}
|\Delta^{u_\eps}_\mathrm{cross}| 
\leq \eps\tau  C_\delta \Vert \nabla u_\eps \Vert_{\rmL^2(\oo_\eps)} \Vert \theta_\eps \Vert_{\rmL^2(\oo_\eps)} \Vert \varphi \Vert_{\rmL^2(\oo_\eps)}. 
\end{align}
Applying the unfolding operators $\T$ and $\Tb$, in particular, rewriting  $\T d_\eps = \mathbb{D}$ and $(\Tb v_\eps)(x,y) = \mathbb{V}(\eps([x/\eps]{+}y),y)$, and using the properties in Lemma \ref{lemma:prop-T} gives
\begin{align}
\label{eq:class-eps-2}
\int_{\oo\times Y_*} \T \dot{u}_\eps \T\varphi \dd x \dd y
& = \int_{\oo\times Y_*} {-} \mathbb{D} \T (\nabla u_\eps) \T(\nabla \varphi) + R(\T u_\eps)\T\varphi \dd x\dd y \nonumber\\
&\quad +\int_{\oo\times \partial T} \left(  a \Tb u_\eps + b \mathbb{V} \right) \Tb \varphi \dd x\dd\sigma (y)  + \Delta^{u_\eps}_\mathrm{cross,app} , 
\end{align}
wherein we replaced the boudnary term $\Tb v_\eps$ with $\mathbb{V}$ and created the approximation error
\begin{align*}
\Delta^{u_\eps}_\mathrm{app} :=
\int_{\oo\times \partial T} b \left(\Tb v_\eps {-} \mathbb{V} \right) \Tb \varphi \dd x\dd\sigma (y).
\end{align*}
Using that $|x {-} \eps([x/\eps] {+} y)| \leq \eps \mathrm{diam}(Y)$ holds for all $(x,y) \in \oo\times \partial T$, we obtain the pointwise estimate $|(\Tb v_\eps)(x,y) {-} \mathbb{V}(x,y)| \leq \eps C \Vert \nabla_x \mathbb{V} \Vert_{\rmL^\infty(\oo)}$ thanks to the Lipschitz continuity of $x \mapsto \mathbb{V}(x,y)$. Together with embedding \eqref{eq:emb} we obtain for the approximation error 
\begin{align}
\label{eq:class-eps-app}
\left\vert \Delta^{u_\eps}_\mathrm{app} \right\vert \leq \eps C \Vert \T \varphi \Vert_{\rmL^2(\oo; \rmH^1(Y_*))}.
\end{align}

We choose the test function $\varphi  =  u_\eps - (u|_{\oo_\eps} {+} \eps \G U) \in \rmH^1(\oo_\eps)$ in \eqref{eq:class-eps-2} such that
\begin{align}
\label{eq:class-eps-3}
&\int_{\oo\times Y_*} \T \dot{u}_\eps (\T u_\eps {-} u) \dd x \dd y \nonumber\\
& = \int_{\oo\times Y_*} {-} \mathbb{D} \T (\nabla u_\eps) [\T(\nabla u_\eps) {-} (\nabla u + \nablaY U)]
+ R(\T u_\eps) (\T u_\eps {-} u) \dd x\dd y \nonumber\\
&\quad  +\int_{\oo\times \partial T} \left(  a \Tb u_\eps + b \mathbb{V} \right) (\Tb u_\eps {-} u) \dd x\dd\sigma (y)  
+ \Delta^{u_\eps}_\mathrm{cross,app,fold}, 
\end{align}
where we added $\pm u$ and $\pm [\nabla u {+} \nablaY U]$, and created the folding mismatch
\begin{align}
\label{eq:class-eps-fold}
\Delta^{u_\eps}_\mathrm{fold} 
& := \int_{\oo\times Y_*} \big\lbrace [R(\T u_\eps) - \T \dot{u}_\eps] (u {-} \T u + \eps\T\G U) \nonumber \\
&\hspace{50pt} {-} \mathbb{D} \T (\nabla u_\eps) [ \nabla u {+} \nablaY U - \T(\nabla (u {+} \eps \G U)) \big\rbrace \dd x\dd y \nonumber\\
&\quad +\int_{\oo\times \partial T} \left(  a \Tb u_\eps + b \mathbb{V} \right) (u {-} \T u + \eps\T\G U) \dd x\dd\sigma (y). 
\end{align}
Exploiting the higher regularity $u\in\rmH^2(\oo)$ we obtain with Lemma \ref{lemma:unfold-err} and Theorem \ref{thm:fold-mismatch}
\begin{align*}
&\Vert \T u {-} u \Vert_{\rmL^2(\oo\times Y_*)} 
+ \Vert \T (\nabla u) {-} \nabla u \Vert_{\rmL^2(\oo\times Y_*)} 
\leq \eps C \Vert u\Vert_{\rmH^2(\oo)} , \\
& \Vert \T [\eps \nabla (\G U)] {-} \nabla_y U \Vert_{\rmL^2(\oo\times Y_*)} \leq \eps C \Vert U \Vert_{\rmH^1(\oo; \rmH^1(Y_*))} .
\end{align*}
The boundary term in \eqref{eq:class-eps-fold} is controlled via
\begin{align}
\label{eq:class-boundary}
&\Vert \T u {-} u \Vert_{\rmL^2(\oo\times \partial T)} 
\leq C_\mathrm{emb} \Vert \T u {-} u \Vert_{\rmL^2(\oo; \rmH^1(Y_*))} \nonumber \\
& =  C_\mathrm{emb} \left( \Vert \T u {-} u \Vert^2_{\rmL^2(\oo\times Y_*)} + \Vert \eps \T(\nabla u) \Vert^2_{\rmL^2(\oo\times Y_*)} \right)^{1/2}
\leq \eps C \Vert u\Vert_{\Hh},
\end{align}
while noting that $\nablaY u = 0$. With this, $\eps \Vert \T\G U\Vert_{\rmL^2(\oo; \rmH^1(Y_*))} \leq \eps 2 \Vert U\Vert_{\rmL^2(\oo; \rmH^1(Y_*))}$ by \eqref{eq:norm-ident}, H\"older's inequality, and the boundedness of $u_\eps$ in \eqref{eq:sol-bound}, we obtain
\begin{align}
\label{eq:class-eps-fold-est}
\int_0^T \left\vert \Delta^{u_\eps}_\mathrm{fold}  \right\vert \dd t \leq O(\eps) .
\end{align}

\emph{Step 2: Reformulation of $u$-equation.} The weak formulation reads 
\begin{align*}
\int_{\oo} \dot{u} \varphi \dd x  = \int_{\oo} - d_{\mathrm{eff}} \nabla u \cdot \nabla \varphi + R(u)\varphi  + \tfrac{|\partial T|}{|Y_*|} (a u + b v_0) \varphi \dd x
\end{align*}
for all $\varphi \in \Hh$ which is equivalent to
\begin{align*}
\int_{\oo\times Y_*} \dot{u} \varphi \dd x\dd y
& = \int_{\oo\times Y_*} {-} \mathbb{D} [\nabla u {+} \nablaY U] \cdot [\nabla \varphi {+} \nabla_y \Phi] +  R(u)\varphi  \dd x \dd y \\
&\quad + \int_{\oo\times\partial T} (a u + b \mathbb{V}) \varphi \dd x\dd\sigma (y) 
\end{align*}
for all test functions $\varphi$ and $\Phi \in \rmL^2(\oo; \rmH^1_\mathrm{per}(Y_*))$. We choose $\Phi_\eps \in \rmL^2(\oo; \rmH^1_\mathrm{per}(Y_*))$ such that we can control the periodicity defect of $\T\varphi$ as in Theorem \ref{thm:per-defect-2}
\begin{align}
\label{eq:class-lim-1}
\int_{\oo\times Y_*} \dot{u} \T\varphi \dd x\dd y
& = \int_{\oo\times Y_*} {-} \mathbb{D} [\nabla u {+} \nablaY U] \cdot \T (\nabla \varphi) +  R(u)\T\varphi  \dd x \dd y \nonumber \\
&\quad + \int_{\oo\times\partial T} (a u + b \mathbb{V}) \Tb\varphi  \dd x\dd\sigma (y) 
+ \Delta^u_\mathrm{per}.
\end{align}
The periodicity defect is given via
\begin{align}
\label{eq:class-lim-per}
\Delta^u_\mathrm{per} 
& := \int_{\oo\times Y_*} [R(u) - \dot{u}] (\varphi {-} \T\varphi) 
{-} \mathbb{D} [\nabla u {+} \nablaY U] \cdot [\nabla \varphi {+} \nablaY \Phi_\eps - \T (\nabla \varphi)]  \dd x \dd y \nonumber \\
&\quad + \int_{\oo\times\partial T} (a u + b \mathbb{V}) (\varphi {-} \Tb\varphi)  \dd x\dd\sigma (y) 
\end{align}
and it is controlled by applying Lemma \ref{lemma:unfold-err}, Theorem \ref{thm:per-defect-2} with $\mathbb{D} [\nabla u {+} \nablaY U] \in \rmH^1(\oo; \rmL^2(Y_*))$, arguing as in \eqref{eq:class-boundary} for the boundary term, and using the boundedness of $u$ in \eqref{eq:sol-bound-limit} via
\begin{align}
\label{eq:class-lim-per-est-0}
\left\vert \Delta^u_\mathrm{per} \right\vert \leq \sqrt {\eps} C \Vert \varphi \Vert_{\Hh} .
\end{align}

Now, we choose $\varphi = \wt u_\eps - (u {+} \eps \widetilde{\G U})$, where $\wt \bullet$ denotes the extension from $\rmH^1(\oo_\eps)$ to $\Hh$ according to Theorem \ref{thm:extension}. Note that the test function belongs to the space $\Hh$ which differs from  Step 1 wherein it belonged to $\rmH^1(\oo_\eps)$. Indeed it holds $\T \wt u_\eps = \T u_\eps$ almost everywhere in $\oo\times Y_*$. Inserting $\varphi$ into \eqref{eq:class-lim-1} and rearranging gives
\begin{align}
\label{eq:class-lim-2}
\int_{\oo\times Y_*} \dot{u} (\T u_\eps {-} u) \dd x\dd y
& = \int_{\oo\times Y_*} \big\lbrace {-} \mathbb{D} [\nabla u {+} \nablaY U] \cdot [\T(\nabla u_\eps) - (\nabla u {+} \nablaY U)] \nonumber \\
&\hspace{140pt} +  R(u)(\T u_\eps {-} u) \big\rbrace \dd x \dd y \nonumber \\
&\quad + \int_{\oo\times\partial T} (a u + b \mathbb{V}) (\Tb u_\eps {-} u)  \dd x\dd\sigma (y) 
+ \Delta^u_\mathrm{per,fold} 
\end{align}
and another folding mismatch
\begin{align}
\Delta^u_\mathrm{fold}
&:= \int_{\oo\times Y_*} \big\lbrace [R(u) - \dot{u}] (u {-} \T u + \eps\G U)  \nonumber \\
&\hspace{60pt} {-} \mathbb{D} [\nabla u {+} \nablaY U] \cdot [\nabla u {+} \nablaY U - \T(\nabla u {+} \eps\nabla \G U)] \big\rbrace \dd x\dd y \nonumber \\
&\quad + \int_{\oo\times\partial T} (a u + b \mathbb{V}) (u {-} \T u + \eps\G U)  \dd x\dd\sigma (y) .
\end{align}
The folding mismatch $\Delta^u_\mathrm{fold}$ has the same form as $\Delta^{u_\eps}_\mathrm{fold}$ in \eqref{eq:class-eps-fold} when replacing $u$ with $\T u_\eps$.

Finally, we control the norm
\begin{align*}
\Vert \varphi \Vert^2_{\Hh} 
& = \Vert \T u_\eps - \T(u {+} \eps\G U) \Vert^2_{\rmL^2(\oo\times Y_*)} 
+ \Vert \T (\nabla u_\eps) - \T( \nabla u {+} \eps\nabla \G U) \Vert^2_{\rmL^2(\oo\times Y_*)} \\
& \leq \Vert \T u_\eps {-} u\Vert^2_{\rmL^2(\oo\times Y_*)} 
+ \Vert \T (\nabla u_\eps) - ( \nabla u {+} \nablaY U) \Vert^2_{\rmL^2(\oo\times Y_*)}  + O(\eps^2)
\end{align*}
by using once more Lemma \ref{lemma:unfold-err} and Theorem \ref{thm:fold-mismatch}. Applying Young's inequality with $\eta_1 >0$ in \eqref{eq:class-lim-per-est-0} yields
\begin{align}
\label{eq:class-lim-per-est}
\left\vert \Delta^u_\mathrm{per} \right\vert 
\leq \eps C_{\eta_1} + \eta_1 \left( \Vert \T u_\eps {-} u\Vert^2_{\rmL^2(\oo\times Y_*)} 
+ \Vert \T (\nabla u_\eps) - ( \nabla u {+} \nablaY U) \Vert^2_{\rmL^2(\oo\times Y_*)} \right).
\end{align}

\emph{Step 3: Derivation of Gr\"onwall-type estimates.}
Subtracting equation \eqref{eq:class-lim-2} from \eqref{eq:class-eps-3} yields
\begin{align*}
\frac12 \frac{\dd}{\dd t} \Vert \T u_\eps {-} u \Vert^2_{\rmL^2(\oo\times Y_*)}
& = \int_{\oo\times Y_*} \big\lbrace {-} \mathbb{D} [\T(\nabla u_\eps) - (\nabla u {+} \nablaY U)] \cdot [\T(\nabla u_\eps) - (\nabla u {+} \nablaY U)] \\
& \hspace{50pt} + [R(\T u_\eps) - R(u)] (\T u_\eps {-} u) \big\rbrace \dd x\dd y \\
& \quad + \int_{\oo\times\partial T} a |\Tb u_\eps {-} u |^2 \dd x\dd\sigma (y) 
 + \Delta^{u_\eps}_\mathrm{cross,app,fold} - \Delta^u_\mathrm{per,fold}.
\end{align*}
Using the uniform ellipticity of $\mathbb{D}$, the Lipschitz continuity of $R$, and the estimations of the periodicity defect in \eqref{eq:class-lim-per-est} gives
\begin{align*}
\frac12 \frac{\dd}{\dd t} \Vert \T u_\eps {-} u \Vert^2_{\rmL^2(\oo\times Y_*)}
& \leq - C_\mathrm{elip} \Vert \T(\nabla u_\eps) - (\nabla u {+} \nablaY U) \Vert^2_{\rmL^2(\oo\times Y_*)} 
+ L \Vert \T u_\eps {-} u \Vert^2_{\rmL^2(\oo\times Y_*)} \\
&\quad + a C_\mathrm{emb} \Vert \T u_\eps {-} u \Vert^2_{\rmL^2(\oo; \rmH^1(Y_*))} 
+ |\Delta^{u_\eps}_\mathrm{cross,app,fold}|+ |\Delta^u_\mathrm{fold}|
+ \eps C_{\eta_1} \\
&\quad 
+ \eta_1 \left( \Vert \T u_\eps {-} u\Vert^2_{\rmL^2(\oo\times Y_*)} 
+ \Vert \T (\nabla u_\eps) - ( \nabla u {+} \nablaY U) \Vert^2_{\rmL^2(\oo\times Y_*)} \right) .
\end{align*}
Choosing $\eta_1 = C_\mathrm{elip}/2$ and integrating over $(0,t)$ with $0<t\leq T$ we get
\begin{align}
\label{eq:class-gronwall-final}
&\Vert \T u_\eps(t) {-} u(t) \Vert^2_{\rmL^2(\oo\times Y_*)}
+ C_\mathrm{elip} \Vert \T(\nabla u_\eps) - (\nabla u {+} \nablaY U) \Vert^2_{\rmL^2((0,t)\times\oo\times Y_*)}  \nonumber \\
&\leq + C \left\lbrace \Vert \T u_\eps {-} u \Vert^2_{\rmL^2((0,t)\times\oo\times Y_*)} + \Vert \T \theta_\eps {-} \Theta \Vert^2_{\rmL^2((0,t)\times\oo\times Y_*)} \right\rbrace \nonumber \\
&\quad + \Vert \T u_\eps^0 {-} u^0 \Vert^2_{\rmL^2(\oo\times Y_*)}
+ O(\eps).
\end{align}

\noindent\textbf{Final step.} 
We add \eqref{eq:slow-gronwall-final} and \eqref{eq:class-gronwall-final} and finally obtain
\begin{align*}
& \Vert \T u_\eps(t) {-} u (t) \Vert^2_{\rmL^2(\oo\times Y_*)}
+\Vert \T \theta_\eps(t) {-} \Theta (t) \Vert^2_{\rmL^2(\oo\times Y_*)} \\
& + C_\mathrm{elip} \left\lbrace \Vert \T (\nabla u_\eps) {-} (\nabla u {+} \nablaY U) \Vert^2_{\rmL^2((0,T)\times\oo\times Y_*)} 
+  \Vert \T (\eps\nabla \theta_\eps) {-} \nablaY \Theta \Vert^2_{\rmL^2((0,T)\times\oo\times Y_*)} \right\rbrace \\
&\leq C \left\lbrace \Vert \T u_\eps {-} u \Vert^2_{\rmL^2((0,T)\times\oo\times Y_*)} + \Vert \T \theta_\eps {-} \Theta \Vert^2_{\rmL^2((0,T)\times\oo\times Y_*)} \right\rbrace + O(\eps) \\
&\quad + \Vert \T u_\eps^0 {-} u^0 \Vert^2_{\rmL^2(\oo\times Y_*)}
+\Vert \T \theta_\eps^0 {-} \Theta^0 \Vert^2_{\rmL^2(\oo\times Y_*)} .
\end{align*}
The application of Gr\"onwall's Lemma and the convergence of the initial values in \eqref{eq:initial-est} complete the proof of \eqref{eq:error-est}. 
\end{proof}

\section{Discussion}
\label{sec:discuss}

Our corrector estimates generalize the qualitative homogenization result obtained in \cite{KAM2014} in two ways: on the one hand we prove quantitative estimates. On the other hand, we consider slow thermal diffusion as well as different scalings $\eps^\alpha$ and $\eps^\beta$ of the cross-diffusion terms.
Under slightly more general assumptions on the data with respect to the $x$-dependence, our estimates imply in particular the rigorous but qualitative homogenization limit for this system.

\medskip
\textbf{What is the limit for arbitrary $\alpha, \beta \geq 0$?} 
For all $\alpha \geq 1$ the limiting $u$-equation remains as it is and the cross-diffusion $\tau\eps^\alpha \nabla u_\eps \cdot \nablaD \theta_\eps$ disappears in the limit $\eps\to0$. For $\alpha = 0$ we have a priori that $\theta_\eps \wto \dashint_{Y_*} \Theta\dd y$ weakly in $\rmL^2(\oo)$ and we expect the additional term $\tau \nabla u \cdot \nablaD \dashint_{Y_*} \Theta\dd y$ in the limit. The choice $\alpha \in (0,1)$ is not meaningful, since the cross-diffusion term is unbounded with $\sim \eps^\alpha$.

For $\beta > 1$ the cross-diffusion term $\mu \eps^\beta \nabla\theta_\eps \cdot \nablaD u_\eps$ vanishes in the limiting $\Theta$-equation and for $\beta < 1$ it diverges with $\eps^{\beta-1}\Vert \eps\nabla \theta_\eps \Vert_{\rmL^2(\oo_\eps)}$. Indeed only the choice $\beta =1$ is meaningful, since it corresponds to the scaling of $\eps^2 \kappa$. 

\medskip
\textbf{Possible generalizations concerning the data.} 
Our analysis allows for not-exactly periodic coefficients such as $d_\eps (x) := \mathbb{D}(x,x/\eps)$ with $\mathbb{D} \in \mathrm{W}^{1,\infty}(\oo; \rmL^\infty(Y_*))$ as in \cite{Rei2016}. The coefficients $\tau_\eps$ and $\mu_\eps$ as well as the reaction term $R_\eps$ may also be not-exactly periodic in the same manner. Moreover all coefficients may additionally depend Lipschitz continuously on time.

The sink/source term $v_\eps$ may be less regular by choosing $v_\eps(t,x) : = [\F \mathbb{V} (t, \cdot,\cdot)](x)$ to capture possible spatial discontinuities in $\mathbb{V} \in \mathrm{C}([0,T]; \rmH^1(\oo; \rmL^2(\partial T)))$. 

On the boundary $\partial T_\eps$ we may consider globally Lipschitz continuous reaction terms $g : \rr \to\rr$. In this case, the boundary term in \eqref{eq:slow-gronwall} is controlled by $L \Vert \Tb \theta_\eps - \Theta \Vert^2_{\rmL^2(\oo\times\partial T)}$, where $L > 0$ denotes the global Lipschitz constant. Non-linear boundary terms may require better initial values to derive the $\rmL^2$-regularity of the time derivatives as in \cite{FMP12,Diss-SR}, however the error estimates hold as they are.

\medskip
\textbf{On the choice of the initial values.} For given $u^0 \in \Hh \cap \rmL^\infty_+(\oo)$ the obvious choice is $u_\eps^0 = u^0|_{\oo_\eps}$ such that the assumption $\Vert \T u_\eps^0 {-} u^0 \Vert_{\rmL^2(\oo\times Y_*)} \leq \sqrt{\eps} C_0$ is satisfied. Perturbations of the form $u_\eps^0 = u^0 {+} \eps V(x,x/\eps)$, which preserve non-negativity, are possible as well. 

In the case of slow diffusion such a direct choice is not possible mainly because  $\theta_\eps^0$ and $\Theta^0$ live in spaces of dimension $d$ and $2d$, respectively. Let $\Theta^0 \in \rmH^1(\oo; \rmH^1_\mathrm{per}(Y_*)) \cap \rmL^\infty_+(\oo\times Y_*)$ be given. One possible choice is $\theta_\eps^0 = \G \Theta^0$, however we are not able to prove $\theta_\eps^0 \in \rmL^\infty_+(\oo_\eps)$ in this case, since  $\rmL^\infty_+(\oo_\eps)$ is not a Hilbert space. Hence, we assume strong differentiability, such as $\Theta^0 \in \mathrm{C}^1(\overline{\oo}; \rmH^1_\mathrm{per}(Y_*) )$ or $\Theta^0 \in \rmH^1(\oo; \mathrm{C}^1_\mathrm{per}(Y_*))$, so that $\theta_\eps^0 = \Theta^0(x,x/\eps)$ is well-defined in $\rmH^1(\oo_\eps) \cap \rmL^\infty_+(\oo)$.

\appendix
\section{Properties of periodic unfolding} 
\label{app:AppendixA}

We recall elementary properties for the periodic unfolding operator $\T$ and the boundary unfolding operator $\Tb$ as well as extensions operators.

\begin{lemma}
\label{lemma:prop-T}
Let $1\leq p, q \leq \infty$ with $1/p + 1/q \leq 1$. 
\begin{compactenum}
\item The operators $\T$ and $\Tb$ are linear and bounded.
\item The product rule 
\begin{align}
\label{eq:product-rule}
\T(uv) = (\T u)(\T v) 
\quad\text{and}\quad
\Tb(uv) = (\Tb u)(\Tb v) 
\end{align}
holds for all $u\in\rmL^{p}(\oo_\eps)$, $v \in\rmL^{q}(\oo_\eps)$ and $u\in\rmL^{p}(\partial T_\eps)$, $v \in\rmL^{q}(\partial T_\eps)$, respectively.
\item The norms are preserved via 
\begin{align}
\label{eq:norm-pres}
\Vert \T u\Vert_{\rmL^p(\oo\times Y_*)} = \Vert u \Vert_{\rmL^p(\oo_\eps)}
\quad\text{and}\quad
\Vert \T u\Vert_{\rmL^p(\oo\times\partial T)} = \sqrt{\eps} \Vert u \Vert_{\rmL^p(\partial T_\eps)} 
\end{align}
for all $u \in \rmL^p(\oo_\eps)$ and $u \in \rmL^p(\partial T_\eps)$, respectively.
\item One has the integration formulas
\begin{align}
\label{eq:int-formula}
\int_{\oo_\eps} u \dd x = \int_{\oo\times Y_*} \T u \dd x\dd y 
\quad\text{and}\quad
\eps \int_{\partial T_\eps} u \dd \sigma(x) = \int_{\oo\times \partial T} \Tb u \dd x\dd\sigma(y) 
\end{align}
for all $u \in \rmL^1(\oo_\eps)$ and $u \in \rmL^1(\partial T_\eps) $, respectively.
\item If $u \in \rmH^1(\oo_\eps)$, then it is $\T u \in \rmL^2(\oo; \rmH^1(Y_*))$ with $\T(\eps\nabla u) = \nablaY (\T u)$.
\item For all $u \in \rmL^p(\oo_\eps)$ it holds $\T[R(u)] = R(\T u)$ where $R: \rr \to \rr$ is an arbitrary function.
\end{compactenum}
\end{lemma}
\begin{proof}
Assertions 1.-5.\ follow from \cite[Prop.~2.5 \& 5.2]{CDZ06} and 6.\ from $[\T R(u)] (x,y) = R[u(\eps[x/\eps] {+} \eps y)] = [T(\T u)](x,y)$ for all $(x,y) \in \oo\times Y_*$.
\end{proof}

\begin{theorem}[{\cite{HB2014}}]
\label{thm:extension}
Under the Assumptions \ref{assump:domain} on the domain there exists a family of linear operators $\mathcal{L}_\eps : \rmH^1(\oo_\eps) \to \Hh$ such that for every $u \in \rmH^1(\oo_\eps)$ it holds
\begin{align*}
(\mathcal{L}_\eps u )|_{\oo_\eps} = u
\quad\text{and}\quad
\Vert \mathcal{L}_\eps u \Vert_{\Hh} \leq C \Vert u \Vert_{\rmH^1(\oo_\eps)} ,
\end{align*}
where $C>0$ only depends on the domains $\oo$, $Y$, and $T$. 
\end{theorem}

\section{Proof of the folding mismatch} 
\label{app:AppendixB}

The proof of Proposition \ref{thm:fold-mismatch} for the folding mismatch follows \cite{Diss-SR, Rei2016} and is adapted to perforated domains. 
We define the scale-splitting operator $\Q$ by $\mathcal{Q}_1$-Lagrangrian interpolants, as customary in finite element methods (FEM), following \cite[Sec.~3]{CDZ06}. By the Sobolev extension theorem, there exists for every $w \in \rmH^1(\oo)$ and $U \in \rmH^1(\oo; \rmH^1_\mathrm{per}(Y_*))$ a function $\wt w \in \rmH^1(\rr^d)$ and $\wt U \in \rmH^1(\rr^d; \rmH^1_\mathrm{per}(Y_*))$, respectively, such that it holds
\begin{align*}
\Vert \wt w \Vert_{\rmH^1(\rr^d)} \leq C \Vert w \Vert_{\rmH^1(\oo)}
\quad\text{and}\quad
\Vert \wt U \Vert_{\rmH^1(\rr^d; \rmH^1_\mathrm{per}(Y_*))} \leq C \Vert U \Vert_{\rmH^1(\oo; \rmH^1_\mathrm{per}(Y_*))} ,
\end{align*}
where $C>0$ only depends on the domain $\oo$. Then $\Q : \rmH^1(\rr^d) \to \mathrm{W}^{1,\infty}(\oo_\eps)$ is given via:

\begin{compactitem}
\item For every node $\eps \xi_k \in \eps \mathbb{Z}^d$ we define
\begin{align*}
(\Q \wt w)(\eps \xi_k) := \dashint_{\eps(\xi_k + Y_*)} \wt w(z) \dd z ,
\end{align*}
(Note that this definition is slightly different than in \cite{CDZ06}. Therein the average is taken over balls $B_\eps$ centered at $\eps \xi_k$ and not touching the pores $T_\eps$. The present definition has the advantage that the equality $(\F \wt w)(\eps \xi_k) = (\Q \wt w)(\eps \xi_k)$ holds for all nodes.)
\item We define $\Q^* w$ on the whole $\rr^d$ by interpolating the nodal values $(\Q \wt w)(\eps \xi_k)$ with $\mathcal{Q}_1$-Lagrangian interpolants yielding polynomials of degree $d$, for more details see \cite[Def.~4.1]{CDG08} or \cite[Def.~2.3.6]{Diss-SR}.
\item On $\oo_\eps$, we set $\Q \wt w:= (\Q^* \wt w)|_{\oo_\eps}$.
\end{compactitem}

For given two-scale functions $U(x,y) = w(x) z(y)$ of product form, we can now construct approximating sequences in $\rmH^1(\oo_\eps)$ via $u_\eps(x) = (\Q \wt w)(x) z(x/\eps)$ and require only the minimal regularity $w \in \rmH^1(\oo)$ and $z \in \rmH^1_\mathrm{per}(Y_*)$. 
According to \cite[Prop.~3.1]{CDZ06} the macroscopic interpolants satisfy
\begin{align}
\label{eq:est-Q}
\Vert \Q \wt w \Vert_{\rmH^1(\oo_\eps)} \leq C \Vert w \Vert_{\rmH^1(\oo)} ,
\end{align}
where $C>0$ only depends on $\oo$ and $Y_*$. Furthermore, we can control the difference between $\F$ and $\Q$ by:
\begin{lemma}
For $w \in \rmH^1(\oo)$ and $z \in \rmL^2(Y_*)$ it holds
\begin{align}
\label{eq:interpol-est}
\Vert ( \F w - \Q \wt w ) z(\tfrac{\cdot}{\eps}) \Vert_{\rmL^2(\oo_\eps)} 
\leq \eps C \Vert w \Vert_{\rmH^1(\oo)} \Vert z \Vert_{\rmL^2(Y_*)} ,
\end{align}
where $C>0$ only depends on $\oo$ and $Y_*$.
\end{lemma}
\begin{proof}
The proof follows along the lines of \cite[Lem.~2.3.7]{Diss-SR} by replacing $Y$ with $Y_*$. 
\end{proof}

Having collected all necessary ingredients, we can now handle the folding mismatch.
\begin{proof}[Proof of Theorem \ref{thm:fold-mismatch}]
The proof follows along the lines of \cite[Thm.~3.4]{Rei2016} by replacing $Y$ with $Y_*$. In a first step estimate \eqref{eq:fold-est} is derived for $U(x,y) = w(x)z(y)$ using the ``folded function'' $\vartheta_\eps(x) = (\Q \wt w)(x)z(x/\eps)$ and the estimates \eqref{eq:est-Q}--\eqref{eq:interpol-est}. In a second step this result is generalized to arbitrary two-scale functions $U(x,y)$ by exploiting the tensor product structure of the space $\rmH^1(\oo;\rmH^1_\mathrm{per}(Y_*))$ and expressing $U(x,y) = \sum_{i=1}^\infty u_i(x) \Phi_i(y)$ in terms of an orthonormal basis $\lbrace \Phi_i \rbrace_{i=1}^\infty \subset \rmH^1_\mathrm{per}(Y_*)$ with $u_i(x) = \int_{Y_*} U(x,y) \Phi_i(y) \dd y$.
\end{proof}

%\textbf{Acknowledgments.} The research of S.R.\ was supported by \emph{Deutsche Forschungsgesellschaft} within  
%\emph{Collaborative Research Center 910: Control of self-organizing nonlinear systems: Theoretical methods and concepts of application}   
% via the project \emph{A5: Pattern formation in systems with multiple scales}.   A.M. thanks {\em NWO MPE ``Theoretical estimates of heat losses in geothermal wells'' (grant nr. 657.014.004)} for funding.

\footnotesize
%\bibliographystyle{my_alpha}
%\bibliography{Bib_Sina} 

\end{document}